\documentclass[centertags,12pt]{amsart}\usepackage{amsthm,verbatim}
\usepackage{amssymb,latexsym}
\usepackage{amsmath}
\usepackage{graphicx}
\usepackage{color}
\usepackage[dvipsnames]{xcolor}
\usepackage{cancel}
\usepackage{soul}
\usepackage{enumerate}
\usepackage{cite}
\newtheorem{theorem}{Theorem}[section]
\newtheorem{definition}[theorem]{Definition}
\newtheorem{remark}[theorem]{Remark}
\newtheorem{lemma}[theorem]{Lemma}
\newtheorem{corollary}[theorem]{Corollary}
\newtheorem{proposition}[theorem]{Proposition}

\def\maria#1 {\mathbb{F}box {\mathbb{F}ootnote {\ }}\ \mathbb{F}ootnotetext { From Maria: {\color{red}#1}}}
\def\daniele#1 {\mathbb{F}box {\mathbb{F}ootnote {\ }}\ \mathbb{F}ootnotetext { From Daniele: {\color{blue}#1}}}
\def\massimo#1 {\mathbb{F}box {\mathbb{F}ootnote {\ }}\ \mathbb{F}ootnotetext { From Massimo: {\color{violet}#1}}}

\begin{document}
\title{New examples of maximal curves with low genus}
\author[D. Bartoli]{Daniele Bartoli}
\address{Dipartimento di Matematica e Informatica, 
Universit\`a degli Studi di Perugia,
via Vanvitelli 1, 06123 Perugia, Italy}
\email{daniele.bartoli@unipg.it}

\author[M. Giulietti]{Massimo Giulietti}
\address{Dipartimento di Matematica e Informatica, 
Universit\`a degli Studi di Perugia,
via Vanvitelli 1, 06123 Perugia, Italy}
\email{massimo.giulietti@unipg.it}

\author[M. Kawakita]{Motoko Kawakita}
\address{Division of Mathematics, Shiga University of Medical Science, Seta Tsukinowa-cho, Otsu City, Shiga, 520-2192 Japan}
\email{kawakita@belle.shiga-med.ac.jp}

\author[M. Montanucci]{Maria Montanucci}
\address{Department of Applied Mathematics and Computer Science, Technical University of Denmark, Asmussens Allé, 2800 Kgs. Lyngby, Denmark}
\email{marimo@dtu.dk}

\date{\today}
\maketitle

\begin{abstract}
Explicit equations for algebraic curves with genus $4$, $5$, and $10$ that are either maximal or minimal over the finite field with $p^2$ elements are obtained for infinitely many $p$'s. 
The key tool is the investigation of their Jacobian decomposition.
Lists of small $p$'s for which maximality holds are provided. In some cases we also describe the automorphism group of the curve.
\end{abstract}

\section{Introduction}

Throughout the paper, by  curve we mean a projective, non-singular, 
geometrically irreducible algebraic curve defined over a finite field $\mathbb{F}_{q^2}$ of order 
$q^2$. 
A curve $\mathcal{X}$ with genus $g=g(\mathcal{X})$ is called {\em $\mathbb{F}_{q^2}$-maximal} if the number of its 
$\mathbb{F}_{q^2}$-rational points attains the Hasse-Weil upper bound, that is,
   $$ 
   |\mathcal{X}(\mathbb{F}_{q^2})|=q^2+1+2qg\, . 
   $$ 
   
Apart from being mathematical objects with intrinsic interest,  maximal curves are often
used for applications in Coding Theory,
Cryptography, and Finite Geometry. For a survey on maximal curves we refer to \cite[Chapter 10]{HKT}.  
For a given $q$, the largest $g$ for which there exists an $\mathbb F_{q^2}$-maximal curve of genus $g$ is $q(q-1)/2$ \cite{Ih}, and equality holds if and only if 
the curve is the Hermitian curve $\mathcal{H}_q$ with equation $X^{q+1}=Y^q+Y$ \cite{RS}.    

A complete solution of the classification problem for maximal curves seems to be out of reach and looking for new examples is still a very active line of research. 
 A  well-known construction method for maximal curves  is based on a  result of Kleiman \cite{Kleiman},  sometimes attributed to Serre (see \cite{Lachaud}), stating that any non-singular curve which is $\mathbb{F}_{q^2}$-covered by an $\mathbb{F}_{q^2}$-maximal curve is also $\mathbb{F}_{q^2}$-maximal. 
Concrete examples of $\mathbb{F}_{q^2}$-maximal curves which are Galois-covered by $\mathcal{H}_{q}$ can be found e.g. in \cite{GSX,CKT,MZq,MX}. 
For a long time the problem of establishing whether all $\mathbb{F}_{q^2}$-maximal curves were covered by $\mathcal H_{q}$ remained open; see e.g. \cite{Geer}. The problem was solved in \cite{GK}, where examples of $\mathbb{F}_{q^2}$-maximal curves, with $q=p^{3h}>8$,  $p$ a prime,  non-covered by  $\mathcal{H}_q$ were constructed. Applying Kleiman-Serre covering result to these curves and to their generalizations \cite{GGS2010}, \cite{BM} and \cite{Skabe}  provided further examples of maximal curves; see \cite{DO,BM2,FG2,ABB,GQZ,GMQZ}. Other recent constructions can be found in \cite{KTT,TT,TTT,TT2,TT3}.

The aim of this paper is to provide new examples of $\mathbb{F}_{p^2}$-maximal curves with low genus. We remark that an $\mathbb{F}_{p^2}$-maximal curve is also $\mathbb F_{p^{2h}}$-maximal for every positive odd integer $h$.
By a criterion going back to Tate \cite[Theorem 2(d)]{TATE} and explicitly pointed out by Lachaud \cite[Proposition 5]{Lachaud}, a curve defined over $\mathbb F_{p}$ with genus $g$ is $\mathbb{F}_{p^2}$-maximal if and only if its Jacobian is $\mathbb F_{p^2}$-isogenous to the $g$-th power of an $\mathbb F_{p^2}$-maximal elliptic curve. Here, such a criterion is applied to plane curves whose Jacobian decomposition was originally studied in the zero characteristic case.
In this direction, our main achievements are 
Theorem \ref{Genus4:mainTheorem}, Remark \ref{remark},
  Theorem \ref{Genus5:mainTheorem}, and Theorem
\ref{Th:Massimalita}. 
These results rely on a theorem by Kani and Rosen, which provides a decomposition of the Jacobian of a curve under some conditions on its automorphism group. 
\begin{theorem}\cite[Theorem B]{KR}\label{Th:KR}
For a curve $\mathcal{X}$, let $G \leq Aut_{\mathbb F_q}(\mathcal{X})$ be a finite group such that $G = H_1 \cup \cdots \cup  H_m$, where the subgroups $H_i$ satisfy $H_i \cap H_j =\{1_{G}\}$ if $i\neq j$. Then we have the following $\mathbb F_q$-isogeny relation
$$J_{\mathcal{X}}^{m-1}\times J^g_{\mathcal{X}/G} \simeq J_{\mathcal{X}/H_1}^{h_1}\times \cdots \times J_{\mathcal{X}/H_m}^{h_m},$$
where $g = |G|$, $h_i = |H_i|$, and, as usual, $J^r=J\times \cdots J$ ($r$ times).
\end{theorem}
\begin{corollary}\cite[Example 1]{KR}\label{Ex:KR}
With the same notation as in Theorem \ref{Th:KR}, if $G$ is the direct product of two cyclic groups of order $2$, say  $G=\{1_G,\alpha_1,\alpha_2,\alpha_3\}$, then
$$J_{\mathcal{X}}\times J^2_{\mathcal{X}/G} \simeq J_{\mathcal{X}/\langle \alpha_1\rangle}\times J_{\mathcal{X}/\langle \alpha_2\rangle}\times J_{\mathcal{X}/\langle \alpha_3\rangle}.$$
\end{corollary}
Theorem \ref{Th:KR}  has already been used e.g. in \cite{AAMZ,Howe,Howe2,Kawakita,Kawakita2,Kawakita3,Kawakita4,NIITSUMA,Bouw} to describe examples of curves with many points with respect to their genus, including maximal curves. 
For some curves investigated in this work,  Theorem \ref{Th:KR}  gives a decomposition of their Jacobian  into elliptic curves. In order to establish whether these factors are isogeneous or not, we will use modular polynomials of low degrees, for which we refer to \cite{Modular1} and the web database \cite{Modular2}.

The full $\mathbb F_{p^2}$-automorphism group $G$ is an invariant of great relevance for an $\mathbb F_{p^2}$-maximal curve $\mathcal X$. In fact, by Kleiman-Serre covering result, for each subgroup $H$ of  $G$ the quotient curve $\mathcal X/H$ is still $\mathbb{F}_{p^2}$-maximal. Therefore, automorphism groups of some of the curves under investigation are determined.

 The paper is organized as follows. In Section \ref{Section:Genus4} we study the curve $\mathcal C$ with equation $$\frac{4(x^2-x+1)^3}{27x^2(x-1)^2}+\frac{4(y^2-y+1)^3}{27y^2(y-1)^2}=1.$$
 In characteristic zero, $\mathcal C$ is one of the two genus $4$ curves with exactly $72$ automorphisms and its Jacobian decomposes as the $4$-th power of an elliptic curve; see \cite{wolfart}. We prove that such a curve is $\mathbb{F}_{p^2}$-maximal for infinitely many primes $p$; see Theorem \ref{Genus4:mainTheorem} and Remark \ref{remark}. We show that its full automorphism group is isomorphic to $(S_3\times S_3)\rtimes C_2$, as it happens in characteristic zero; see Section \ref{Genus4:FullAutomorphismGroup}. 
 
 In Section \ref{Section:Genus5} we deal with a family of curves with genus $5$, given by Equation \eqref{family}. In characteristic zero, these curves were characterized by Shaska \cite{shas} as the genus $5$ hyperelliptic curves with automorphism group isomorphic to the direct product of $A_4$ and a group of order $2$; Paulhus later observed that their Jacobian decomposes in the product of three isogenous elliptic curves and a hyperelliptic curve of genus $2$; see \cite[Proposition 4]{PAULHUS}. The main result of the section is Theorem \ref{Genus5:mainTheorem}, which states that some curves in the family are either $\mathbb{F}_{p^2}$-maximal or $\mathbb{F}_{p^2}$-minimal (that is, attaining the Hasse-Weil lower bound) for infinitely many primes $p$. A list of $p$ for which specific curves of the family are $\mathbb{F}_{p^2}$-maximal is provided; see Theorem \ref{Th:Massimalita}. 
In Section \ref{Section:Genus5Characterization} we extend Shaska's characterization result to the positive characteristic case; see Theorem \ref{caratterizzazione}. 

Finally, in Section \ref{Section:HigherGenus} we report some results on the Jacobian decomposition and the $\mathbb{F}_{p^2}$-maximality  of other curves with genus  $10$.




\section{A curve of genus $4$}\label{Section:Genus4}
In this section we study the plane curve $\mathcal{C}$ of the affine equation

\begin{equation}\label{Equation:Genus4}
\frac{4(X^2-X+1)^3}{27X^2(X-1)^2}+\frac{4(Y^2-Y+1)^3}{27Y^2(Y-1)^2}=1
\end{equation} 
over the finite fields $\mathbb{F}_{p}$, $p> 7$ prime.


\subsection{$\mathbb F_{p^2}$-maximality of $\mathcal{C}$}
Let $F=\mathbb{F}_{p^2}(x,y)$ with $(4(x^2-x+1)^3)/(27x^2(x-1)^2)+(4(y^2-y+1)^3)/(27y^2(y-1)^2)=1$ be the function field of the curve $\mathcal{C}$. Let $\tau_1$, $\tau_2$, $\sigma$ be the order-$2$ automorphisms of $F$ defined by $\sigma(x,y)=(y,x)$, $\tau_1(x,y)=(-x+1,y)$, $\tau_2(x,y)=(x,-y+1)$. We consider the group $G=\langle \tau_1,\tau_2,\sigma\rangle$. Such a group is elementary abelian of order $8$ and can be written as 
\begin{equation}\label{equation:GruppoG}
G=\langle \tau_1\rangle \cup \langle \tau_2\rangle  \cup \langle \sigma\rangle \cup \langle \sigma \tau_1\rangle\cup \langle \sigma \tau_1\tau_2\rangle,
\end{equation}
where $\langle \sigma \tau_1\rangle=\{id, \sigma \tau_1,\tau_1\tau_2,\sigma \tau_2\}$ and $|\langle \tau_1\rangle|=|\langle \tau_2\rangle|=| \langle \sigma\rangle|=|\langle \sigma \tau_1\tau_2\rangle|=2$.


\begin{proposition} \label{tau1}
The Jacobian variety of both the quotient curves  $\mathcal{C}/\langle \tau_1 \rangle$ and $\mathcal{C}/\langle \tau_2 \rangle$ is $\mathbb{F}_p$-isogenous to twice the Jacobian variety of the elliptic curve with equation   
$$Y^2 = X^3 + \frac{9}{16}X^2 + \frac{3}{16}X + \frac{1}{64}.$$
\end{proposition}
\proof
Note that $\tau_1$ fixes both $y_1 = y$ and $x_1=x(-x+1)$. As $\tau_1$ is an involution, the function field of $\mathcal C/\tau_1$ over $\overline{\mathbb{F}}_p$ is $\mathbb F_p$-isomorphic to $\overline{\mathbb{F}}_p(x_1,y_1)$. A straightforward computation gives 
$$4 (  (-y_1+1 )^3 x_1^2 (x_1-1)^2 +(x_1^2-x_1+1 )^3 y_1^2)-27 y_1^2 x_1^2 (x_1-1 )^2 )=0.$$
Consider the genus $2$ hyperellittic function field $\overline{\mathbb{F}}_p(u,v)$ with
\begin{equation}\label{hyperellittic}
v^2  = 4u^6 - 12u^5 + 21u^4 - 22u^3 + 21u^2 - 12u + 4,
\end{equation}
and let $\mathcal{G}: Y^2  = 4X^6 - 12X^5 + 21X^4 - 22X^3 + 21X^2 - 12X + 4$ be the corresponding hyperelliptic curve of genus 2. 
Then $\overline{\mathbb{F}}_p(x_1,y_1)\cong \overline{\mathbb{F}}_p(u,v)$ via the $\mathbb{F}_{p}$-isomorphism 
$\psi:\overline{\mathbb{F}}_p(u,v)\to \overline{\mathbb{F}}_p(\xi,\eta)$ defined by

$$\psi(u,v)=\left(\frac{a_1(u,v)}{a_3(u,v)},\frac{a_2(u,v)}{a_3(u,v)}\right),$$ with 

\begin{eqnarray*}
a_1(u,v)&=&\frac{1}{2}u^8 - \frac{3}{2}u^7 + \frac{21}{8}u^6 - \frac{1}{4}u^5v - \frac{13}{4}u^5 + \frac{3}{8}u^4v + 
    \frac{17}{4}u^4\\&&
    + \frac{1}{8}u^3v - \frac{13}{4}u^3 + \frac{1}{8}u^2v + \frac{21}{8}u^2 + \frac{3}{8}uv 
    - \frac{3}{2}u - \frac{1}{4}v + \frac{1}{2},\\
a_2(u,v)&=& u^6 - \frac{3}{2}u^5 + u^4 + \frac{1}{2}u^3v - \frac{1}{2}u^3 - \frac{1}{2}u^2v,\\
a_3(u,v)&=&u^6 - 2u^5 + 2u^4 - 2u^3 + u^2.\\
\end{eqnarray*} 
Hence $\mathcal{G}\cong\mathcal{C}/\langle \tau_1 \rangle$ and applying the same argument to $x_2=x$, $y_2=y(-y+1)$, and $\tau_2$, we see  that $\mathcal{G}\cong\mathcal{C}/\langle \tau_2 \rangle$ as well.

Now we investigate the Jacobian decomposition of the hyperellittic curve $\mathcal{G}$. We apply  Corollary \ref{Ex:KR} to $\mathcal{G}$ and the automorphism group of order $4$ generated by the involutions
$$\psi_1(u,v)=(u,-v), \qquad \psi_2 (u,v)= (-u+1 ,v).$$
It is immediate to see that the quotient curve $\mathcal{G}/\langle \psi_1\rangle$ is rational, whereas both $\mathcal{G}/\langle \psi_2\rangle$ and $\mathcal{G}/\langle \psi_1\psi_2\rangle$ are elliptic. In fact,
     $\psi_2$ fixes $u_1=-u^2+u$ and $v_1=v$, and 
    $${v_1}^2=-4{u_1}^3+9{u_1}^2-12{u_1}+4$$
    holds; also,
    $\psi_1\psi_2$ fixes $u_2=u_1$ and $v_2=v(2u-1)$ and 
    $${v_2}^2=(-4{u_2}^3+9{u_2}^2-12{u_2}+4)(-4{u_2}+1).$$
The function fields of the curves
$\mathcal{G}/\langle \psi_2\rangle$ and 
$\mathcal{G}/\langle \psi_1\psi_2\rangle$  are
$\mathbb{F}_p$-isomorphic to $\overline{\mathbb{F}}_p(\mathcal{C}_1)=\overline{\mathbb{F}}_p(x^\prime,y^\prime)$ and $\overline{\mathbb{F}}_p(\mathcal{C}_2)=\overline{\mathbb{F}}_p(x^{\prime\prime},y^{\prime\prime})$ with

$$(y^{\prime})^2 = (x^{\prime})^3 + \frac{9}{16}(x^{\prime})^2 + \frac{3}{16}(x^{\prime}) + \frac{1}{64}$$
and 
$$(y^{\prime\prime})^2 = (x^{\prime\prime})^3 + \frac{11}{12}(x^{\prime\prime})^2 + \frac{1}{9}(x^{\prime\prime}) + \frac{1}{81},$$

respectively, where 
$$ x^{\prime}=-\frac{1}{4}u_1, \quad y^{\prime}=-\frac{1}{16}v_1, \qquad x^{\prime\prime}=\frac{1/6 u_2 - 1/24}{-(u_2)^2 + 1/2u_2 - 1/16}, \quad y^{\prime\prime}=\frac{1/36v_2}{-(u_2)^2 + 1/2u_2 - 1/16}.$$ 
The two elliptic curves $\mathcal{C}_1$ and $\mathcal{C}_2$ are isogenous via the isogeny defined over their function field by $\theta(x^{\prime\prime},y^{\prime\prime})=(\theta_1(x^{\prime\prime},y^{\prime\prime}),\theta_2(x^{\prime\prime},y^{\prime\prime}))$ where
\begin{eqnarray*}
    \theta_1&=&\frac{1/4(x^{\prime\prime})^3 - 1/9x^{\prime\prime}}{(x^{\prime\prime} + 2/9)^2},\\
    \theta_2&=&\frac{1/8(x^{\prime\prime})^3y^{\prime\prime} + 1/12(x^{\prime\prime})^2y^{\prime\prime} + 
    1/18(x^{\prime\prime})y^{\prime\prime} - 1/81y^{\prime\prime}}{(x^{\prime\prime} + 2/9)^3}.\\
\end{eqnarray*}
By Corollary \ref{Ex:KR}, the claim follows. 

\endproof

\begin{proposition} \label{sigma}

The Jacobian varieties of both the quotient curves  $\mathcal{C}/\langle \sigma\rangle$ and $\mathcal{C}/\langle \sigma\tau_1\tau_2\rangle$ are $\mathbb{F}_p$-isogenous to the Jacobian variety of the elliptic curve of equation   $$Y^2 = X^3 + \frac{9}{16}X^2 + \frac{3}{16}X + \frac{1}{64}.$$

\end{proposition}
\proof
First note that the curves $\mathcal{C}/\langle \sigma\rangle$ and $\mathcal{C}/\langle \sigma\tau_1\tau_2\rangle$ are $\mathbb F_p$-isomorphic since the  subgroups $\langle \sigma\rangle$ and $\langle \sigma\tau_1\tau_2\rangle$ are conjugated in the $\mathbb F_p$-automorphism group of $\mathcal C$. Indeed $\gamma \sigma = \tau_1\tau_2\gamma$, where $\gamma(x,y)=(x,-y+1)$ is an automorphism  of $F$. Therefore, we are going to deal only with the curve $\mathcal{C}/\langle \sigma\rangle$.
 
Both $x_2=x+y$ and $y_2=xy$ are fixed by $\sigma$.
Then $\overline{\mathbb{F}}_p(x_2,y_2)\cong \overline{\mathbb{F}}_p(\mathcal{C}/\langle \sigma\rangle)$.
Also,
$$
\begin{array}{rl}
4 x_2^4 y_2^2 + 4 x_2^4 - 8 x_2^3 y_2^3 - 12 x_2^3 y_2^2 - 12 x_2^3 y_2 - 8 x_2^3 + 4 x_2^2 y_2^4 +   8 x_2^2 y_2^3 + 21 x_2^2 y_2^2 + 8 x_2^2 y_2\\
+ 4 x_2^2 + 12 x_2 y_2^4 + 14 x_2 y_2^3 + 14 x_2 y_2^2 +  12 x_2 y_2 - 8 y_2^5 - 19 y_2^4 - 38 y_2^3 - 19 y_2^2 - 8 y_2&=0.\\
\end{array}
$$
As $\sigma$ fixes exactly $6$ points of $\mathcal C$, by the Riemann-Hurwitz formula the genus of $\mathcal{C}/\langle \sigma\rangle$ is $1$. 

An isogeny $\varphi$ between $\mathcal{C}/\langle \sigma \rangle$ (with homogeneous coordinates $X_1$, $Y_1$ and $Z_1$) and the elliptic curve 
$$Y^2 Z=X^3+\frac{9}{16} X^2 Z+\frac{3}{16} X Z^2+\frac{1}{64} Z^3$$ can be constructed as follows. Consider the map $\varphi_1$ given by
{\footnotesize \begin{eqnarray*}\varphi_1(X_2)&=&2^5 \cdot7^4 X_2^3 Y_2^5 Z_2 - 2^4\cdot 3 \cdot7^4 X_2^2 Y_2^6 Z_2 - 2^4 \cdot7^4 X_2 Y_2^7 Z_2 + 2^5\cdot 7^4 Y_2^8 Z_2 +
    2^3 \cdot 7^4 X_2^3 Y_2^4 Z_2^2- 2^3\cdot 7^5 X_2^2 Y_2^5 Z_2^2\\
    &&+ 2^3 \cdot7^4 X_2 Y_2^6 Z_2^2 + 2^3\cdot 5\cdot 7^4 Y_2^7 Z_2^2 +
    2^6\cdot 7^4\cdot X_2^3 Y_2^3 Z_2^3 - 2^4 7^5 X_2^2 Y_2^4 Z_2^3 - 2^3\cdot 3^2\cdot 7^4 X_2 Y_2^5 Z_2^3 + 2^5\cdot 7^5 Y_2^6 Z_2^3
     \\
    &&+ 2^3 \cdot7^4 X_2^3 Y_2^2 Z_2^4- 2^4\cdot 7^5 X_2^2 Y_2^3 Z_2^4 - 2 \cdot7^4\cdot 19 X_2 Y_2^4 Z_2^4 + 2\cdot 7^4\cdot 97 Y_2^5 Z_2^4
    + 2^5\cdot 7^4 X_2^3 Y_2 Z_2^5 - 2^3\cdot 7^5 X_2^2 Y_2^2 Z_2^5 \\
    &&- 2^3\cdot 3^2\cdot 7^4 X_2 Y_2^3 Z_2^5 + 2\cdot 3^2 \cdot 7^4\cdot 17 Y_2^4 Z_2^5-2^4 3 7^4 X_2^2 Y_2 Z_2^6 + 2^3\cdot 7^4 X_2 Y_2^2 Z_2^6 +  2^3 \cdot 7^5 Y_2^3 Z_2^6 - 2^4\cdot 7^4 X_2 Y_2 Z_2^7\\
    &&+
    2^3\cdot 3\cdot 5\cdot 7^4Y_2^2 Z_2^7,\\
\varphi_1(Y_2)&=&-2^7\cdot 7^6 X_2^3 Y_2^6 + 2^8 \cdot7^6 X_2^2 Y_2^7 - 2^7\cdot 7^6 X_2 Y_2^8 + 2^6\cdot 3 \cdot7^6 X_2 Y_2^7 Z_2 -
    2^6\cdot 3 \cdot7^6 Y_2^8 Z_2 - 2^5\cdot 7^6\cdot 11 X_2^3 Y_2^4 Z_2^2\\
    &&+ 2^5 \cdot7^6\cdot 29 X_2^2 Y_2^5 Z_2^2-
    2^5\cdot 7^6\cdot 11 X_2 Y_2^6 Z_2^2 - 2^5\cdot 7^5\cdot 127 Y_2^7 Z_2^2 + 2^6\cdot 3\cdot 7^6 X_2^3 Y_2^3 Z_2^3 -
    2^6\cdot 7^6\cdot X_2^2 Y_2^4 Z_2^3 \\
    &&- 2^5\cdot 3\cdot 7^6 X_2 Y_2^5 Z_2^3- 2^5\cdot 5\cdot 7^5\cdot 13 Y_2^6 Z_2^3 -
    2^5\cdot 7^6\cdot 11 X_2^3 Y_2^2 Z_2^4 + 2^9\cdot 7^6\cdot X_2^2 Y_2^3 Z_2^4 - 2^3 \cdot7^6\cdot 67 X_2 Y_2^4 Z_2^4 \\
    &&-
    2^3\cdot 3^3 7^5\cdot 13 Y_2^5 Z_2^4 + 2^6 \cdot3\cdot 7^6 X_2^3 Y_2 Z_2^5- 2^5\cdot 3\cdot 7^6 X_2^2 Y_2^2 Z_2^5 +
    2^4\cdot 5\cdot 7^6 X_2 Y_2^3 Z_2^5 + 2^3\cdot 7^5\cdot 167 Y_2^4 Z_2^5\\
    && - 2^7 \cdot7^6 X_2^3 Z_2^6 + 2^6\cdot 7^6 X_2^2 Y_2 Z_2^6- 2^5\cdot 3 \cdot7^6 X_2 Y_2^2 Z_2^6 - 2^7\cdot 3\cdot 7^5 13 Y_2^3 Z_2^6 + 2^7\cdot 7^6\cdot X_2^2 Z_2^7 +
    2^6 \cdot 7^5\cdot 41 Y_2^2 Z_2^7\\ &&- 2^6 \cdot7^5\cdot 13 Y_2 Z_2^8,\\
\varphi_1(Z_2)&=&Y_2^7 Z_2^2 - 3/2 Y_2^6 Z_2^3 + 15/4 Y_2^5 Z_2^4 + 1/4 Y_2^4 Z_2^5 + 15/4 Y_2^3 Z_2^6 - 3/2 Y_2^2 Z_2^7+ Y_2 Z_2^8.
\end{eqnarray*}}
Then define $$\varphi(X_2:Y_2:1)=(1/(2^6 7^4) \varphi_1(X_2) : 1/(2^9 7^6) \varphi_1(Y_2) - 1/(2^8 7^4)\varphi_1(X_2) + 3^2/(2^2 7) : 1).$$ 
By direct checking with MAGMA $\varphi$ is an isogeny between the two elliptic curves over the field of rationals. Clearly this map provides an isogeny over $\mathbb{F}_{p}$ for every $p$ such that no coefficients of $\varphi_1$ vanish. It has been checked with MAGMA that also for these special values of $p$ the two elliptic curves are $\mathbb{F}_p$-isogenous.
\endproof



\begin{proposition} \label{tau1tau2}
The Jacobian variety of the curve  $\mathcal{C}/\langle \tau_1\tau_2 \rangle$ is $\mathbb{F}_p$-isogenous to twice the Jacobian variety of the elliptic curve of equation   $$Y^2 = X^3 + \frac{9}{16}X^2 + \frac{3}{16}X + \frac{1}{64}.$$
\end{proposition}

\proof
Note that  
$$x_3=x(-x+1)+y(-y+1), \qquad y_3=\frac{x}{y}+\frac{-x+1}{-y+1}$$
are both fixed by the involution $\tau_1\tau_2$.
They satisfy $L(u,v)=0$, where 
$$
\begin{array}{ll}
L(x_3,y_3):=&16  x_3^6   y_3^4 - 40  x_3^5   y_3^4 - 32  x_3^5   y_3^3 + 32  x_3^5   y_3^2 + 32  x_3^4   y_3^6 - 80  x_3^4   y_3^5
        + 185  x_3^4   y_3^4 + 80  x_3^4   y_3^3 + 128  x_3^4   y_3^2 \\
        {}&         - 32  x_3^4   y_3+ 16  x_3^4 -  40  x_3^3   y_3^6 - 124  x_3^3   y_3^5 - 24  x_3^3   y_3^4 - 482  x_3^3   y_3^3 - 518  x_3^3   y_3^2 - 
        112  x_3^3   y_3 - 104  x_3^3 \\
        {}&         + 16  x_3^2   y_3^8- 48  x_3^2   y_3^7+ 212  x_3^2   y_3^6 -    352  x_3^2   y_3^5 + 1184  x_3^2   y_3^4 - 916  x_3^2   y_3^3 + 2193  x_3^2   y_3^2 + 318  x_3^2   y_3 
       \\ 
       {}&
       + 425x_3^2 - 32x_3 y_3^7 + 48 x_3 y_3^6 - 296  x_3   y_3^5 + 304  x_3   y_3^4 - 1076  x_3   y_3^3 + 524  x_3   y_3^2 - 1312  x_3   y_3 \\
       {}& - 832  x_3       + 32   y_3^6 - 64   y_3^5 + 308   y_3^4 - 488   y_3^3 + 1000   y_3^2 - 1024   y_3 + 1024.
\end{array}
$$

To show that $\overline{\mathbb{F}}_p(x_3,y_3)$ with $L(x_3,y_3)=0$ is the function field of $\mathcal{C}/\langle \tau_1 \tau_2 \rangle$ it is sufficient to note that
$\overline{\mathbb{F}}_p(x_3,y_3)$ has genus $2$. Indeed, from the Hurwitz genus formula and Theorem \ref{Ex:KR} applied with respect to $G$, $\mathcal{C}/ \langle \sigma \tau_1 \rangle$ has genus $1$ and since $\tau_1 \tau_2 \in \langle \sigma \tau_1 \rangle$, $$g(\mathcal{C}/ \langle \sigma \tau_1 \rangle)<g(\mathcal{C}/\langle \tau_1 \tau_2 \rangle) \leq 2.$$ Also, $\overline{\mathbb{F}}_p(x_3,y_3)$ is $\mathbb{F}_p$-isomorphic to $\overline{\mathbb{F}}_p(u',v')$, with

\begin{equation}\label{hyperellittic1}
 {v'}^2 + ({u'}^2 + u')v' = -{u'}^6 - 3{u'}^5 - {u'}^4 - 7{u'}^3 -{u'}^2 - 3{u'} - 1,
 \end{equation}
   
via the isomorphism $\phi(x_3,y_3)=\left(\frac{b_1(x_3,y_3)}{b_3(x_3,y_3)},\frac{b_2(x_3,y_3)}{b_3(x_3,y_3)}\right)$ given by 
\begin{eqnarray*}
b_1(x_3,y_3)&=&-2 x_3^4  + x_3^3   -2 x_3^2,\\
   b_2(x_3,y_3)&=& \frac{2x_3^5 -3 x_3^4  + x_3^3  -2 x_3^2 y_3  - 3 x_3^2  -x_3 y_3   - y_3  -1}{2},\\
b_3(x_3,y_3)&=& x_3^5 + 2 x_3^3  + x_3 .\\
\end{eqnarray*}
Now, $\overline{\mathbb{F}}_p(u',v')$ is $\mathbb{F}_p$-isomoprhic to $\overline{\mathbb{F}}_p(u^{\prime \prime},v^{\prime \prime})$, with 
$$
(v^{\prime\prime})^2=  -4(u^{\prime \prime})^6 - 12(u^{\prime \prime})^5 - 3(u^{\prime \prime})^4 - 26(u^{\prime \prime})^3 - 3(u^{\prime \prime})^2 -   12(u^{\prime \prime}) - 4,
$$
via the isomorphism $\theta(u',v')=\left(\frac{c_1(u',v')}{c_3(u',v')},\frac{c_2(u',v')}{c_3(u',v')}\right)$
\begin{eqnarray*}
c_1(u',v')&=&{u'}^2,\\
   c_2(u',v')&=& {u'}^2+u'+2v',\\
c_3(u',v')&=& {u'}^3.\\
\end{eqnarray*}
Consider the involutory automorphisms of the hyperelliptic function field $\overline{\mathbb{F}}_p(u^{\prime \prime},v^{\prime \prime})$ given by
$$\psi_1(u^{\prime \prime},v^{\prime \prime})=(1/u^{\prime \prime},v^{\prime \prime}/(u^{\prime \prime})^3) \qquad \psi_2 (u^{\prime \prime},v^{\prime \prime})= (u^{\prime \prime},-v^{\prime \prime}).$$
 The fixed fields of $\overline{\mathbb{F}}_p(u^{\prime \prime},v^{\prime \prime})$ with respect to  $\psi_1,\psi_2,\psi_1\psi_2$ are
$$\overline{\mathbb{F}}_p(s,t_1), \,\,\overline{\mathbb{F}}_p(u^{\prime \prime}),\,\,\overline{\mathbb{F}}_p(s,t_2),$$
respectively, where $s=u^{\prime \prime}+1/u^{\prime \prime}$, $t_1=v^{\prime \prime}+v^{\prime \prime}/(u^{\prime \prime})^3$, $t_2=v^{\prime \prime}-v^{\prime \prime}/(u^{\prime \prime})^3$ and 

\begin{equation}\label{Curva1}
t_1^2 - 4s^6 - 12s^5 + 21s^4 + 26s^3 - 51s^2 + 24s - 4=0,
\end{equation}

\begin{equation}\label{Curva2}
-t_2^2 - 4s^6 - 12s^5 + 21s^4 + 42s^3 - 3s^2 - 12s + 4=0.
\end{equation}

Consider the function fields $\overline{\mathbb{F}}_p(s^{\prime},t_1^{\prime})$, $\overline{\mathbb{F}}_p(s^{\prime\prime},t_1^{\prime\prime})$  with 
$$(t_1^{\prime})^2 =(s^{\prime})^3 - 9517824s^{\prime} + 11448262656$$
and 
$$(t_2^{\prime})^2 = (s^{\prime\prime})^3 - 4541184(s^{\prime\prime}) + 493881753.$$
The function field $\overline{\mathbb{F}}_p(s,t_1)$ is isomorphic to $\overline{\mathbb{F}}_p(s^{\prime},t_1^{\prime})$ 
via the isomorphism 

$$\theta^{\prime}(s,t_1)=\left(d_1(s,t_1)/d_3(s,t_1),d_2(s,t_1)/d_3(s,t_1)\right)$$
where

\begin{eqnarray*}
d_1(s,t_1)&=&2304s^5 + 9216u^4 - 2880s^3 + 1152s^2t_1 - 17856s^2 + 2880st_1 + 11520s + 720t_1 - 2304,\\
d_2(s,t_1)&=&248832s^5 + 995328s^4 + 55296s^3t_1 - 311040s^3 + 207360s^2t_1 - 1928448s^2+ 103680st_1\\
&&+ 27648t_1^2 + 
    1244160s - 55296t_1- 248832,\\
d_3(s,t_1)&=&t_1.\\
\end{eqnarray*}

The function field $\overline{\mathbb{F}}_p(s,t_2)$ is isomorphic to $\overline{\mathbb{F}}_p(s^{\prime\prime},t_2^{\prime})$ via the isomorphism

$$\theta^{\prime\prime}(s,t_2)=\left(d_1^\prime(s,t_2)/d_3^\prime(s,t_2),d_2^\prime(s,t_2)/d_3^\prime(s,t_2)\right)$$
where

\begin{eqnarray*}
d_1^{\prime}(s,t_2)&=&36864s^5t_2 + 184320s^4t_2 + 175104s^3t_2 - 36864s^2t_2 - 4176t_2^3 - 46080st_2 + 18432t_2,\\
d_2^{\prime}(s,t_2)&=&84934656s^5 - 442368s^3t_2^2 + 424673280s^4 - 2654208s^2t_2^2 + 403439616s^3 - 5640192st_2^2\\
&&- 84934656s^2 - 
    10506240t_2^2 - 106168320s + 42467328,\\
d_3^{\prime}(s,t_2)&=&t_2^3.\\
\end{eqnarray*}

 The two elliptic curves $\mathcal E_1: Y^2 =X^3 - 9517824X + 11448262656$ and $\mathcal E_2: Y^2 = X^3 - 4541184X + 493881753$ are isogenous to $Y^2=X^3 + 9/16 X^2 + 3/16 X + 1/64$ via the $\mathbb{F}_{p^2}$-isogenies

$$(X : Y)\mapsto \left(
\frac{f_1(X)}{f_2(X)} : 
    \frac{g(X,Y)}{f_3(X)}\right),$$
    where 
    
     $$\begin{array}{lll}
     f_1(X)&=&-1/20736 X^3 - 1/16 X^2 + 837 X - 1123632, \\
    f_2(X)&=&X^2 -  2592 X + 1679616,\\
   
    g(X,Y)&=&-\xi/2985984 X^3 Y + \xi/768 X^2 Y - 75\xi/16 X Y + 8073\xi Y,\\
    f_3(X) &=&-X^3 + 3888 X^2 - 5038848 X + 2176782336;\\
    \end{array}$$
    
with $\xi^2=-1$ for $\mathcal  E_1$, and

     $$\begin{array}{lll}
     f_1(X)&=&1/20736 X^3 - 11/48 X^2 - 213 X + 720720,\\
    f_2(X)&=&X^2 - 864 X+ 186624\\
   
    g(X,Y)&=&1/2985984 X^3 Y - 1/2304 X^2 Y + 137/48 X Y - 9371 Y,\\
    f_3(X) &=&X^3 - 1296 X^2 + 559872 X - 80621568;\\
    \end{array}$$
for $\mathcal E_2$. From Corollary \ref{Ex:KR} applied to $\overline{\mathbb{F}}_p(u^{\prime \prime},v^{\prime \prime})$ with respect to $\langle \psi_1,\psi_2 \rangle \cong C_2 \times C_2$ the claim follows.
\endproof

\begin{theorem}\label{Genus4:mainTheorem}
Let $p>7$ be a prime. Then the $\mathbb{F}_p$-rational curve $\mathcal{C}$ is 
$\mathbb{F}_{p^2}$-maximal if and only if the elliptic curve 
$$ \mathcal{E} \ : \ Y^2=X^3 + 9/16 X^2 + 3/16 X + 1/64$$
is $\mathbb{F}_{p^2}$-maximal. In particular, $\mathcal C$
is $\mathbb{F}_{p^2}$-maximal for infinitely many primes. 
\end{theorem}
\proof
We apply Theorem \ref{Th:KR} to the curve $\mathcal{C}$ and the group $G=\langle \tau_1,\tau_2,\sigma\rangle$. Such a group can be written as 
$$G=\langle \tau_1\rangle \cup \langle \tau_2\rangle  \cup \langle \sigma\rangle \cup \langle \sigma \tau_1\rangle\cup \langle \sigma \tau_1\tau_2\rangle;$$
see Equation \eqref{equation:GruppoG}.
 In this case, $m=5$, $h_1=h_2=h_3=h_5=2$, $h_4=4$ and 
 \begin{equation} \label{dec}
 J_{\mathcal{C}}^{4}\times J^4_{\mathcal{C}/G} \simeq J_{\mathcal{C}/\langle \tau_1\rangle}^{2}\times J_{\mathcal{C}/\langle \tau_2\rangle}^{2}\times J_{\mathcal{C}/\langle \sigma\rangle}^{2}\times J_{\mathcal{C}/\langle \sigma\tau_1\rangle}^{4}\times J_{\mathcal{C}/\langle \sigma\tau_1\tau_2\rangle}^{2}.
\end{equation}
By the Hurwitz genus formula applied to $G$, we deduce that $\mathcal{C}/G$ is rational. In fact, the number of places of $\mathcal C$ fixed by elemets of $G$ exceeds $6$.
By Propositions \ref{tau1} and \ref{sigma} the Jacobians of $\mathcal{C}/\langle \tau_1 \rangle$, $\mathcal{C}/\langle \tau_2 \rangle$, $\mathcal{C}/\langle \sigma \rangle$ and $\mathcal{C}/\langle \sigma \tau_1 \rangle$ are $\mathbb{F}_p$-isogenous to a power of the Jacobian of $\mathcal{E}$. Hence Equation \eqref{dec} reads 
 \begin{equation} \label{dec1}
 J_{\mathcal{C}} \simeq J_{\mathcal{E}}^{3}\times  J_{\mathcal{C}/\langle \sigma\tau_1\rangle}.
\end{equation}

To prove the first claim, note  that the $\mathbb F_{p^2}$-maximality of $\mathcal C$ implies the $\mathbb F_{p^2}$-maximality of $\mathcal E$ since $\mathcal E$ is $\mathbb{F}_{p^2}$-covered by $\mathcal{C}$.
Assume now that $\mathcal{E}$ is $\mathbb{F}_{p^2}$-maximal.
By Proposition \ref{tau1tau2} the Jacobian of the curve $\mathcal{C}/\langle \tau_1\tau_2\rangle$ is $\mathbb{F}_{p^2}$-isogenous to a power of $J_{\mathcal{E}}$, and hence $\mathcal{C}/\langle \tau_1\tau_2\rangle$ is an $\mathbb F_{p^2}$-maximal curve.  But $\mathcal{C}/\langle \sigma\tau_1\rangle$ is $\mathbb{F}_{p^2}$-covered by $\mathcal{C}/\langle \tau_1\tau_2\rangle$; therefore $\mathcal{C}/\langle \sigma\tau_1\rangle$ is $\mathbb{F}_{p^2}$-maximal too, and so is $\mathcal C$ by the Tate-Lachaud criterion mentioned in the introduction.

An elliptic curve over the complex field with integer coefficients, when it is viewed as a curve defined over $\mathbb F_p$ with $p$ a prime, is supersingular for infinitely many $p$'s; see \cite{Elkies}. Since a supersingular elliptic curve defined over $\mathbb F_p$ is $\mathbb F_{p^2}$-maximal, the elliptic curve $\mathcal E$ is $\mathbb F_{p^2}$-maximal for  infinitely many $p$'s. This proves the second claim. 
\endproof

\begin{remark}\label{remark}
MAGMA  \cite{MAGMA} computations show that the list of primes up to 100000 for which  the curve $\mathcal{E}$, and therefore  $\mathcal{C}$, is $\mathbb{F}_{p^2}$-maximal is
\begin{eqnarray*}
17,71,251,647,827,1889,3527,3617,4409,6569,11969,12113,12527,12689,13913,22031,\\
23039,23633,26297,28871,31769,35171,35729,48527,60497,60623,61487,82457,93383,93761.
\end{eqnarray*}
\end{remark}

\begin{remark}\label{remark:Copertura}
For $p=17$  the curve $\mathcal{C}$ is Galois covered by the Hermitian curve $\mathcal{H}_{17}$ over $\mathbb{F}_{17^2}$. In fact, let  $\mathcal{H}_{17}: Y^{18}=X^{17} Z + Z^{17} X$. Consider the automorphism group $G$ of $\mathcal{H}_{17}$ generated by $\alpha_1,\alpha_2,\alpha_3$ and $\gamma$ where
$$\alpha_1(X,Y,Z)=(-X,\lambda Y,Z), \ \alpha_2(X,Y,Z)=(-\lambda c Z, Y, -\lambda/c X),$$ $$\alpha_3(X,Y,Z)=(cX,Y,-X/c), \ {\rm and} \ \gamma(X,Y,Z)=(X,\mu Y,Z),$$
with $\mu^3=1$, $\lambda^2=-1$ and $2c^2=\lambda+1$. Then $G=\langle \alpha_1,\alpha_2,\alpha_3 \rangle \times \langle \gamma \rangle \cong Q_8 \times C_3$, where $Q_8$ is the quaternion group of order $8$. Using MAGMA we obtained a model for the quotient curve $\mathcal{H}_{17}/G$ over $\mathbb{F}_{17^2}$:
$$\mathcal{H}_{17}/G: \begin{cases} YZ+16XW=0, \\ X^3+4XY^2+16Z^2W+9W^3=0. \end{cases}$$
An $\mathbb{F}_{17^2}$-rational isomorphism between $\mathcal{C}:(4(X^2-XZ+Z^2)^3)(27Y^2(Y-Z)^2)+(4(Y^2-YZ+Z^2)^3)(27X^2(X-Z)^2)=(27Y^2(Y-Z)^2)(27X^2(X-Z)^2)Z^2$ and $\mathcal{H}_{17}/G$ is given by $\phi(X,Y,Z)=(\phi_1(X,Y,Z),\phi_2(X,Y,Z),\phi_3(X,Y,Z),$ $\phi_4(X,Y,Z))$ where
$$\phi_1(X,Y,Z)=w^{159}X^5Y^6 Z + w^{159}X^3Y^8Z + w^{33}X^5Y^5Z^2 + w^{28}X^4Y^6Z^2 +
    w^{231}X^3Y^7Z^2 + w^{89}X^2Y^8Z^2$$
    $$+ w^{141}X^5Y^4Z^3
    + w^{190}X^4Y^5Z^3
    + w^{78}X^3Y^6Z^3 + w^{161}X^2Y^7Z^3 + w^{93}XY^8Z^3 + w^{213}X^5Y^3Z^4
    + w^{135}X^4Y^4Z^4$$
    $$+ 2 X^3Y^5Z^4 + 4X^2Y^6Z^4 + w^{165}XY^7Z^4+
    w^{156}Y^8Z^4 + w^{177}X^5Y^2Z^5 + w^{253}X^4Y^3Z^5 + w^{119}X^3Y^4Z^5$$
    $$+w^{91}X^2Y^5Z^5+ w^{127}XY^6Z^5 + w^{228}Y^7Z^5 + w^{124}X^4Y^2Z^6+
    w^{281}X^3Y^3Z^6 + X^2Y^4Z^6 + w^{52}XY^5Z^6$$
    $$+ 14 Y^6Z^6 +
    w^{209}X^3Y^2Z^7+ w^{101}X^2Y^3Z^7 + w^{64}XY^4Z^7 + w^{64}Y^5Z^7 +
    w^{69}X^3YZ^8 + w^{194}X^2Y^2Z^8$$
    $$+w^{221}XY^3Z^8 + w^{275}Y^4Z^8 +
    w^{177}X^3Z^9 + w^{74}X^2YZ^9 + w^{97}XY^2Z^9 + w^{220}Y^3Z^9 + 3 X^2Z^{10}
    + w^{75}XYZ^{10} $$
    $$+ w^{202}Y^2Z^{10}+ w^{129}XZ^{11} + w^{12}YZ^{11},$$
    
{$$\phi_2(X,Y,Z)=w^{105}X^5 Y^7 + w^{105}X^3 Y^9 + w^{195}X^5 Y^6 Z + w^{52}X^4 Y^7Z +
    w^{33}X^3Y^8Z + 12X^2Y^9Z + w^{267}X^5Y^5Z^2$$
    $$+ w^{142}X^4Y^6Z^2+
    w^{94}X^3Y^7Z^2 + 14 X^2 Y^8 Z^2 + w^{57} X Y^9 Z^2 + w^{51} X^5 Y^4 Z^3 +
    w^{214} X^4 Y^5 Z^3 + w^{132} X^3 Y^6 Z^3$$
    $$+ w^{76} X^2 Y^7 Z^3 + w^{273} X Y^8 Z^3 +
    w^{231} X^5 Y^3 Z^4+ w^{286} X^4 Y^4 Z^4 + w^{38} X^3 Y^5 Z^4 + w^{286} X^2 Y^6 Z^4
    + w^{128} X Y^7 Z^4$$
    $$+ w^{285} X^5 Y^2 Z^5 + w^{178} X^4 Y^3 Z^5 + w^8 X^3 Y^4 Z^5
    + w^{75} X^2 Y^5 Z^5+ w^{207} X Y^6 Z^5 + w^{165} Y^7 Z^5 + w^{232} X^4 Y^2 Z^6$$
    $$+w^{175} X^3 Y^3 Z^6 + w^{187} X^2 Y^4 Z^6 + w^{231} X Y^5 Z^6 + w^{255} Y^6 Z^6 +
    w^{212} X^3 Y^2 Z^7 + w^{145} X^2 Y^3 Z^7 + w^{179} X Y^4 Z^7$$
    $$+ w^{39} Y^5 Z^7 +
    w^{123} X^3 Y Z^8 + w^{71} X^2 Y^2 Z^8 + w^{151} X Y^3 Z^8 + w^{111} Y^4 Z^8 +
    w^{285} X^3 Z^9 + 2 X^2 Y Z^9 $$
    $$+ w^{82} X Y^2 Z^9 + w^3 Y^3 Z^9 + 11 X^2 Z^{10} +
    w^{75} X Y Z^{10} + w^{57} Y^2 Z^{10} + w^{237} X Z^{11},$$}
    
{$$\phi_3(X,Y,Z)=w^{207} X^5 Y^7 + w^{207} X^3 Y^9 + w^9 X^5 Y^6 Z + w^{189} X^4 Y^7 Z +
    w^{135} X^3 Y^8 Z + w^{117} X^2 Y^9 Z + w^{45} X^5 Y^5 Z^2$$
    $$+ w^{279} X^4 Y^6 Z^2+
    w^{225} X^3 Y^7 Z^2 + w^{45} X^2 Y^8 Z^2 + w^{243} X Y^9 Z^2 + w^{45} X^5 Y^4 Z^3 +
    w^{27} X^4 Y^5 Z^3 + w^{171} X^3 Y^6 Z^3$$
    $$+ w^{27} X^2 Y^7 Z^3 + w^{171} X Y^8 Z^3+
    w^9 X^5 Y^3 Z^4 + w^{27} X^4 Y^4 Z^4 + w^{27} X^3 Y^5 Z^4 + w^{207} X^2 Y^6 Z^4 +
    w^{135} X Y^7 Z^4$$
    $$+ w^{207} X^5 Y^2 Z^5 + w^{279} X^4 Y^3 Z^5 + w^{27} X^3 Y^4 Z^5+
    w^9 X^2 Y^5 Z^5 + w^9 X Y^6 Z^5 + w^{167} Y^7 Z^5 + w^{189} X^4 Y^2 Z^6$$
    $$+w^{171} X^3 Y^3 Z^6 + w^9 X^2 Y^4 Z^6 + w^{207} X Y^5 Z^6 + w^{257} Y^6 Z^6 +
    w^{225} X^3 Y^2 Z^7+ w^{207} X^2 Y^3 Z^7$$
    $$+ w^{207} X Y^4 Z^7 + w^{136} Y^5 Z^7 +
    w^{135} X^3 Y Z^8 + w^{27} X^2 Y^2 Z^8 + w^9 X Y^3 Z^8 + w^{171} Y^4 Z^8 +
    w^{207} X^3 Z^9+ w^{45} X^2 Y Z^9$$
    $$+ w^{135} X Y^2 Z^9 + w^{257} Y^3 Z^9 +
    w^{117} X^2 Z^{10} + w^{171} X Y Z^{10} + w^{140} Y^2 Z^{10} + w^{243} X Z^{11} + 4 Y Z^{11},$$}

{$$\phi_4(X,Y,Z)=Y^8Z^4 + 13 Y^7Z^5 + Y^6Z^6 + 11 Y^5Z^7 + 5 Y^4Z^8 + Y^3Z^9 + 10 Y^2Z^{10} + 9YZ^{11},$$
and $w$ is a primitive element of $\mathbb{F}_{17^2}$.
}
We are not able to tell whether $\mathcal C$ is Galois covered by the Hermitian curve for $p>17$.
\end{remark}

\subsection{Automorphism group of $\mathcal{C}$}\label{Genus4:FullAutomorphismGroup}

The aim of this section is to compute the full automorphism group of $\mathcal{C}$ when $p  \geq 5$. First, we construct a subgroup $H$ of automorphisms of $\mathcal{C}$, and then we will prove that when $p \geq 5$, $H$ coincides with the full automorphism group of $\mathcal{C}$.
\begin{remark} \label{Galmeno}
The full automorphism group of $\mathcal{C}$ contains a subgroup 
$$H =(\langle \alpha_1,\alpha_2 \rangle \times \langle \alpha_3,\alpha_4 \rangle) \rtimes \langle \alpha_5 \rangle\cong (S_3 \times S_3) \rtimes C_2, $$ where $$\alpha_1:(x,y) \mapsto \bigg( \frac{-1}{x-1}, y\bigg), \quad \alpha_2: (x,y) \mapsto \bigg(\frac{1}{x},y\bigg)$$ $$\alpha_3:(x,y) \mapsto \bigg(x,  \frac{-1}{y-1}\bigg), \quad \alpha_4: (x,y) \mapsto \bigg(x,\frac{1}{y}\bigg), \quad \alpha_5:(x,y) \mapsto (y,x).$$ It is easily seen that
 $$\alpha_2^2=\alpha_4^2=\alpha_5^2=\alpha_1^3=\alpha_3^3=1$$ 
and
 $\alpha_2 \alpha_1 \alpha_2=\alpha_1^2$, $\alpha_4 \alpha_3 \alpha_4=\alpha_3^2$ so that $\langle \alpha_1,\alpha_2 \rangle \cong \rm{S}_3$ and $\langle \alpha_3,\alpha_4 \rangle \cong \rm{S}_3$, and $\alpha_5$ interchanges $\langle \alpha_1,\alpha_2 \rangle$ into $\langle \alpha_3,\alpha_4 \rangle$ by conjugation.
\end{remark}


\begin{lemma} \label{istame}
 $|Aut(\mathcal{C})| \in \{72,144,216\}$.
\end{lemma}

\begin{proof}
We first prove that the size of $Aut(\mathcal C)$ is coprime with $p$.
Assume on the contrary that $S$ is a Sylow $p$-subgroup of $Aut(\mathcal{C})$. Then by \cite[Theorem 1]{N1987}, 
$$|S| \leq \max \bigg\{4,\frac{4p}{(p-1)^2}16\bigg\}.$$
For  $p \geq 11$ this contradicts $|S|\ge p$.
By direct checking with MAGMA, if either $p=5$ or $p=7$ then the curve $\mathcal{C}$ is ordinary, that is, the $p$-rank $\gamma$ of $\mathcal C$ equals its genus $g(\mathcal{C})=4$. 
If $p=7$, by \cite[Theorem 1 (i)]{N1987} we have $|S| \leq 3p/(p-2) \leq 21/5<7$, a contradiction. 
If $p=5$ then  $|S|=5$. By \cite[Theorem 1.3]{GKlargep}, $|Aut(\mathcal{C})|$ divides $2p(p-1)=40$, a contradiction as $72 \mid |Aut(\mathcal{C})|$ by Remark \ref{Galmeno}.
This shows that $p$ does not divide $| Aut(\mathcal{C})|$. By \cite[Theorem 11.108]{HKT}, $|Aut(\mathcal{C})|=2^a 3^b 5^c 7^d$ for some integers $a,b,c,d \geq 0$. Observe that $a \geq 3$ and $b \geq 2$ as $72 \mid |Aut(\mathcal{C})|$.  Hurwitz bound \cite[Theorem 11.56]{HKT} yields $|Aut(\mathcal{C})| \leq 84(g-1)=252$, hence $|Aut(\mathcal{C})| \in \{72,144,216\}$.
\end{proof}

Next we prove that  the cases $|Aut(\mathcal{C})|=144$ and $|Aut(\mathcal{C})|=216$ cannot actually occur.

\begin{theorem}
Let $p \geq 5$ and let $H$ be as in Remark \ref{Galmeno}.
The full automorphism group $Aut(\mathcal{C})$ coincides with $H$. In particular,
  $Aut(\mathcal{C})\cong (S_3 \times S_3) \rtimes C_2$, where $C_2$ denotes a cyclic group of order $2$ and $S_3$ is the symmetric group on $3$ letters. 
\end{theorem}

\begin{proof}

\begin{itemize}
\item Suppose that $|Aut(\mathcal{C})|=216$. Then $Aut(\mathcal{C})$ is one of the $177$ groups of order $216$ (up to isomorphism) and it contains a subgroup $G$ isomorphic to $(S_3 \times S_3) \rtimes C_2$. By direct checking with MAGMA, we have that $Aut(\mathcal{C}) \cong SmallGroup(216,i)$ for some $i=157,158,159$. We recall that from \cite[Theorem 11.79]{HKT} if $A$ is an abelian subgroup of $Aut(\mathcal{C})$, then $|A| \leq 4g+4=20$. On the other hand, for each $i=157,158,159$ the group $SmallGroup(216,i)$ has an abelian subgroup of order $27$, a contradiction. 

\item Suppose that $|Aut(\mathcal{C})|=144$. In particular, $Aut(\mathcal{C})$ is one of the $197$ groups of order $144$ (up to isomorphism) containing a subgroup isomorphic to $(S_3 \times S_3) \rtimes C_2$. By direct checking with MAGMA, $Aut(\mathcal{C}) \cong SmallGroup(144,i)$ for some $i=182,186$. Since  $SmallGroup(144,186)$ contains an elementary abelian group of order $8$ and $\mathcal{C}$ has even genus,  \cite[Lemma 6.3]{GK2017} provides a contradiction. Hence, $Aut(\mathcal{C}) \cong SmallGroup(144,182)$. Therefore, $Aut(\mathcal{C})$ contains a unique Sylow $3$-subgroup $S_3$, which is elementary abelian of order $9$; also, the elements of order $3$ are all conjugated in $Aut(\mathcal{C})$ whereas they split into $2$ conjugacy classes in the subgroup $H\cong (S_3 \times S_3) \rtimes C_2$. 
Let $S$ denote the unique Sylow $3$-subgroup of $Aut(\mathcal C)$ and let 
$\tilde g$ be the genus of the quotient curve $\mathcal{C}/S$.
By the Hurwitz genus formula \cite[Theorem 11.72]{HKT} applied to
$S$,
%
 either $\tilde g=1$ or $\tilde g=0$.

Assume that $\tilde g=1$. Then from \cite[Theorem 11.72]{HKT},
$$6=2 \cdot a + 8 \cdot b,$$
for some $a \geq 0$ and $b \geq 0$ as $S$ contains just proper subgroups of order $3$. Necessarily, $b=0$ and $a=3$. Geometrically, this shows that for $3$ points, say 
$P_1,P_2$, and $P_3$,  the stabilizer $S_{P_i}$ has order $3$.  Since the $8$ elements of order $3$ in $S$ are all conjugated and therefore fix the same number $n$ of points,  the only possibility is $n=1$ and 
$S_{P_i} \ne S_{P_j}$ for every $i \ne j$ and $i,j=1,2,3$.
Note that $S$ acts transitively on the set $\{P_1,P_2,P_3\}$.
Let $N$ be the normalizer of $S$ in $Aut(\mathcal{C})$. Then by direct checking with MAGMA, $N=S \rtimes M$, for some subgroup $M$ of order $16$.
The group $M$ is isomorphic to a subgroup of automorphisms of the elliptic curve $\mathcal{C}/S$, which fixes the point of $\mathcal{C}/S$ corresponding to the orbit $\{P_1,P_2,P_3\}$. 
But this is impossible by \cite[Theorem 11.94(ii)]{HKT}.

We are left with the case where $\mathcal{C}/S$ is a rational curve.
Note that $S$ cannot have fixed points on $\mathcal{C}$ since it is not a cyclic group; see \cite[Theorem 11.49]{HKT}. By the Hurwitz genus formula we get that
$$6=-2 \cdot 9 +2a,$$
and hence $a=12$ and
each element of order $3$ in $S$ has exactly $3$ fixed points.  
For a subgroup  $T$ of $S$  of order $3$, by the Hurwitz genus formula applied to $T$,  the quotient curve
$\mathcal{C}/T$ is elliptic. The normalizer $N_T$ of $T$ in $Aut(\mathcal{C})$ has order  $36$  and
the
factor group $N_T/T$ contains an elementary abelian subgroup $K$ of order $4$.
Clearly, $N_T$ acts on the set of $3$ points fixed by $T$.  
%
Therefore $K$, viewed as an automorphism group of $\mathcal{C}/T$, acts on the set of $3$ points of $\mathcal{C}/T$ corresponding to the fixed points of $T$. Necessarily $K$ fixes at least one of these $3$ points. Since $K$ is not cyclic we have a contradiction by \cite[Theorem 11.49]{HKT}. 

This shows that the elements of order $3$ cannot be all conjugated in $Aut(\mathcal{C})$ and hence that $Aut(\mathcal{C})$ cannot be isomorphic to $SmallGroup(144,182)$. 
\end{itemize}
By Lemma \ref{istame} $Aut(\mathcal{C})$ has order $72$, and the claim follows by Remark \ref{Galmeno}.
\end{proof}

\begin{remark}
Up to our knowledge, the unique known examples of $\mathbb{F}_{17^2}$-maximal curves of genus $4$ were   
$$y^{12} = x(x+1)^{16}$$
and 
$$ x^6 + y^6 + 2x^2y^2 + 12xy + 14=0;$$
see \cite{RI,Manypoints}.
We point out that the curve $\mathcal{C}$ is not $\overline{\mathbb{F}}_{17}$-isomorphic to either of them, since the automorphism group of the former   is $((C_2\times C_2)\rtimes (C_3\times C_3))\rtimes C_2\not\cong Aut(\mathcal{C})$, whereas the order of the automorphism group of the latter is $12$.
\end{remark}

\section{Curves of Genus $5$}\label{Section:Genus5}
Let $q=p^h$ where $h \in \mathbb{N}$ and $p$ is an odd prime. Let $a \in \mathbb{F}_{q^2}$   be such that 
\begin{equation}\label{Condizione}
    a^2+108 \ne 0.
\end{equation} We consider the following family of algebraic curves

\begin{equation} \label{family}
\mathcal{C}_a: Y^2=X^{12}-aX^{10}-33X^8+2aX^6-33X^4-aX^2+1.
\end{equation}

Let $F_a={\overline{\mathbb F}_p}(x,y)$ with $y^2=x^{12}-ax^{10}-33x^8+2ax^6-33x^4-ax^2+1$ be the corresponding function field. Let
$G=\langle \alpha_1,\alpha_2\rangle$, where 
$\alpha_1$, $\alpha_2$ are the following involutory automorphisms of $F_a$:
$$\alpha_1: (x,y) \rightarrow (x,-y),$$
$$\alpha_2: (x,y) \rightarrow (-x,y).$$

\begin{lemma} \label{lemma1}
Let $q=p^h$ where $h \in \mathbb{N}$ and $p$ is an odd prime. Let $a \in \mathbb{F}_{q^2}$. Then the curve $\mathcal{C}_a$ defined as in \eqref{family} has genus $5$ and it is hyperelliptic. 
\end{lemma}

\begin{proof}
To prove the statement it is sufficient to observe that, if $f(X)=X^{12}-aX^{10}-33X^8+2aX^6-33X^4-aX^2+1$ then $f(X)$ has multiple roots if and only if $a^2+108=0$ since the resultant of $f(X)$ and its derivative $f^\prime(X)$ is equal to $(a^2+108)^8$. Now the claim follows from \cite[Proposition 3.7.3 and Corollary 3.7.4]{Sti}.
\end{proof}


A decomposition of the Jacobian $J_a$ of $\mathcal{C}_a$ is obtained applying Theorem \ref{Th:KR}, with respect to the automorphism group $G=\langle \alpha_1,\alpha_2\rangle$.

\begin{remark} \label{obs1}
Since $\alpha_1$ is the hyperelliptic involution of $\mathcal{C}_a$, the fixed fields $FixG$ and $Fix\langle \alpha_1 \rangle$ are both rational, see \cite[Theorem 11.98]{HKT}. More precisely, it is easily seen that 
%
$Fix\langle \alpha_1 \rangle={\overline{\mathbb F}_p}(x)$ and
$FixG={\overline{\mathbb F}_p}(x^2)$.
\end{remark}

To apply Theorem \ref{Th:KR} we need to compute the fixed fields of $\langle \alpha_2\rangle$ and $\langle\alpha_3\rangle=\langle\alpha_1\alpha_2\rangle$. 

\begin{lemma} \label{lemma2}
Let $q=p^h$ where $h \in \mathbb{N}$ and $p$ is an odd prime. Let $a \in \mathbb{F}_{q^2}$ with $a^2+108 \ne 0$. Then $Fix\langle \alpha_2 \rangle={\overline{\mathbb F}_p}(\eta,\theta)$ where $\eta=y$, $\theta=x^2$ and
\begin{equation} \label{quot1}
\eta^2=\theta^6-a\theta^5-33\theta^4+2a\theta^3-33\theta^2-a\theta+1.
\end{equation}
In particular $Fix\langle \alpha_2 \rangle$ is a hyperelliptic curve of genus $2$.
\end{lemma}

\begin{proof}
The first assertion follows from the fact that $x^2$ is fixed by $\alpha_2$.
%

The polynomial $f_1(X)=X^6-aX^5-33X^4+2aX^3-33X^2-aX+1$ has no multiple roots by Lemma \ref{lemma1}. 
From \cite[Proposition 3.7.3 and Corollary 3.7.4]{Sti}, $Fix\langle \alpha_2\rangle$ is hyperelliptic of genus $2$. 
%
\end{proof}

\begin{lemma} \label{lemma3}
Let $q=p^h$ where $h \in \mathbb{N}$ and $p$ is an odd prime. Let $a \in \mathbb{F}_{q^2}$ with $a^2+108 \ne 0$ and let $\alpha_3=\alpha_1 \alpha_2$ so that $\alpha_3: (x,y) \rightarrow (-x,-y)$. 
Then $Fix\langle \alpha_3 \rangle={\overline{\mathbb F}_p}(\rho,\theta)$ where
$\rho=xy$, $\theta=x^2$ and
\begin{equation} \label{quot2}
\rho^2=\theta^7-a\theta^6-33\theta^5+2a\theta^4-33\theta^3-a\theta^2+\theta.
\end{equation}
In particular $Fix\langle \alpha_3 \rangle$ is a hyperelliptic curve  of genus $3$.
\end{lemma}

\begin{proof}
Since $\alpha_3$ fixes both $\theta$ and $\rho$, we have
 ${\overline{\mathbb F}_p}(\theta,\rho) \subseteq Fix\langle \alpha_3 \rangle$. 
 As the degree of the function field extension 
${\overline{\mathbb F}_p}(x,y):{\overline{\mathbb F}_p}(\theta,\rho)$ is clearly $2$, we have that  $Fix\langle \alpha_3\rangle={\overline{\mathbb F}_p}(\theta,\rho)$.
  By Equation \eqref{family},
$$\rho^2=y^2 \theta=(x^{12}-ax^{10}-33x^8+2ax^6-33x^4-ax^2+1)\theta=\theta^{7}-a\theta^{6}-33\theta^5+2a\theta^4-33\theta^3-a\theta^2+\theta.$$
Consider the polynomial $f_2(X)=X^{7}-aX^{6}-33X^5+2aX^4-33X^3-aX^2+X$. Then $f_2(X)$ has no multple roots over ${\overline{\mathbb F}_p}$ since the resultant of $f_2(X)$ and $f_2^\prime(X)$ is equal to $(a^2+108)^4 \ne 0$. From \cite[Proposition 3.7.3 and Corollary 3.7.4]{Sti}, $Fix\langle \alpha_3\rangle$ is a hyperelliptic function field of genus $3$. 
\end{proof}

At this point, we get a first decomposition of the Jacobian of $\mathcal{C}_a$ as a straightforward application of Corollary \ref{Ex:KR}.

\begin{corollary} \label{KR1}
Let $q=p^h$ where $h \in \mathbb{N}$ and $p$ is an odd prime. Let $a \in \mathbb{F}_{q^2}$ with $a^2+108 \ne 0$. Then the Jacobian variety $J_a$ of the curve $\mathcal{C}_a$ defined in \eqref{family} decomposes as

\begin{equation} \label{decomposition1}
{J_a}\sim {J_{a,2}} \times {J_{a,3}},
\end{equation}
where $J_{a,2}$ is the Jacobian variety of the hyperelliptic curve of genus $2$
\begin{equation} \label{gen2}
\mathcal{C}_{a,2}: Y^2=X^6-aX^5-33X^4+2aX^3-33X^2-aX+1,
\end{equation}
while $J_{a,3}$ is the Jacobian variety of the hyperelliptic curve of genus $3$
\begin{equation} \label{gen3}
\mathcal{C}_{a,3}: Y^2=X^7-aX^6-33X^5+2aX^4-33X^3-aX^2+X.
\end{equation}
\end{corollary}

Now, a decomposition of the Jacobian $J_{a,2}$ of the curve $\mathcal{C}_{a,2}$ 
whose  function field  is defined by Equation \eqref{quot1}), is obtained by applying Theorem \ref{Th:KR} to the automorphism group $G_2=\langle \beta_1,\beta_2\rangle$ with
$$\beta_1: (\theta,\eta) \rightarrow (\theta,-\eta),$$
$$\beta_2: (\theta,\eta) \rightarrow (1/\theta,\eta/\theta^3).$$
It is easily seen that $\beta_1,\beta_2 \in Aut(\mathcal C_{a,2})$ and that $G_2 = \langle \beta_1,\beta_2 \rangle$ is elementary abelian of order $4$. Also, since $\beta_1$ is the hyperelliptic involution of $\mathcal{C}_{a,2}$ we know that both $Fix\langle \beta_1 \rangle={\overline{\mathbb F}_p}(\theta)$ and $FixG_2={\overline{\mathbb F}_p}(\theta+1/\theta)$ are rational.
We will prove that $J_{a,2}$ decomposes completely as a power of the Jacobian of a unique elliptic curve. 
First, we  compute the fixed fields of $\langle\beta_2\rangle$ and $\langle\beta_3\rangle=\langle\beta_1 \beta_2\rangle$. 

\begin{lemma} \label{lemma4}
Let $q=p^h$ where $h \in \mathbb{N}$ and $p$ is an odd prime. Let $a \in \mathbb{F}_{q^2}$ with $a^2+108 \ne 0$. Then $Fix\langle \beta_2 \rangle={\overline{\mathbb F}_p}(\tilde{x},\tilde{y})$ where

\begin{equation} \label{quot3}
\tilde{y}^2=4\tilde{x}^3+(a+6)\tilde{x}^2+(a-6)\tilde{x}-4.
\end{equation}
In particular $Fix\langle \beta_2 \rangle$ is elliptic.
\end{lemma}
\begin{proof} Let
$$\tilde{x}=-\frac{(\theta-1)^2}{(\theta+1)^2} \quad and \quad \tilde{y}=-\frac{\eta}{(\theta+1)^3}.$$
Note that both $\tilde{x}$ and $\tilde{y}$ are fixed by $\beta_2$; also,
the field extension $\overline{\mathbb F}_p(\theta,\eta):\overline{\mathbb F}_p(\tilde{x},\tilde{y})$  has degree $2$, whence $Fix\langle \beta_2 \rangle=\overline{\mathbb F}_p(\tilde{x},\tilde{y})$.
%
%
%
%
%
%
By Equation \eqref{quot1} $\tilde{y}^2=\tilde{x}^3+\frac{a+6}{4}\tilde{x}^2+\frac{a-6}{4}\tilde{x}-1$. The polynomial $X^3+\frac{a+6}{4}X^2+\frac{a-6}{4}X-1$ has no multiple roots since $(a^2+108)\neq 0$ and hence $Fix\langle \beta_2 \rangle$ is elliptic.
The claim follows considering the isomorphism $(\tilde{x},\tilde{y}) \mapsto (\tilde{x},\tilde{y}/2)$.
\end{proof}

\begin{lemma} \label{lemma5}
Let $q=p^h$ where $h \in \mathbb{N}$ and $p$ is an odd prime. Let $a \in \mathbb{F}_{q^2}$ with $a^2+108 \ne 0$. Let $\beta_3=\beta_1 \beta_2$ so that $\beta_3: (\theta,\eta) \mapsto (1/\theta, -\eta/\theta^3)$. Then $Fix\langle \beta_3 \rangle={\overline{\mathbb F}_p}(\overline{x},\overline{y})$ where
\begin{equation} \label{quot4}
\overline{y}^2=4\overline{x}^3+(-a+6)\overline{x}^2+(-a-6)\overline{x}-4.
\end{equation}
In particular $Fix\langle \beta_3 \rangle$ is elliptic.
\end{lemma}

\begin{proof}
%
%
%
Let
$$\overline{x}=-\frac{(\theta+1)^2}{(\theta-1)^2} \quad and \quad \overline{y}=\frac{\eta}{(\theta-1)^3}.$$
Then the proof is analogous to that of Lemma \ref{lemma4}.
\end{proof}

\begin{remark} \label{obsiso}
Let $q=p^h$ where $h \in \mathbb{N}$ and $p$ is an odd prime. Let $a \in \mathbb{F}_{q^2}$ with $a^2+108 \ne 0$. Then the quotient curves $\mathcal{C}_{a,2}/\langle \beta_2 \rangle$ and $\mathcal{C}_{a,2}/\langle \beta_3 \rangle$ are $\mathbb{F}_{q^2}$-isomorphic. The isomorphism between the two curves is given by
$$\varphi: (\tilde x, \tilde y) \mapsto (-\tilde x-1,\xi \tilde y),$$
where $\xi^2=-1$. 
\end{remark}

At this point the complete decomposition of $J_{a,2}$ follows by applying Corollary \ref{Ex:KR}  to $G_2$. 

\begin{corollary} \label{KR2}
Let $q=p^h$ where $h \in \mathbb{N}$ and $p$ is an odd prime. Let $a \in \mathbb{F}_{q^2}$ with $a^2+108 \ne 0$. Then the Jacobian variety $J_{a,2}$ of the curve $\mathcal{C}_{a,2}$ defined in \eqref{gen2} decomposes as $J_{a,2} \sim J_{ E_1}^2$ where $J_{ E_1}$ is the Jacobian variety of the elliptic curve
\begin{equation} \label{primaellittica}
 E_1: Y^2=4X^3+(a+6)X^2+(a-6)X-4.
\end{equation}
Furthermore, the Jacobian variety $J_a$ of the curve $\mathcal{C}_a$ defined in \eqref{family} decomposes as
\begin{equation} \label{decomposition2}
{J_a} \sim J_{ E_1}^2 \times {J_{a,3}},
\end{equation}
where $J_{a,3}$ is as in Corollary \ref{KR1}.
\end{corollary}


We proceed now with the complete decomposition of the Jacobian $J_{a,3}$ of the curve $\mathcal{C}_{a,3}$, whose function field ${\overline{\mathbb F}_p}(\rho,\theta)$ is defined in Equation \eqref{quot2}, by applying  Theorem \ref{Th:KR}, with respect to the automorphism group $G_3=\langle \gamma_1,\gamma_2\rangle$, where
$$\gamma_1: (\theta,\rho) \rightarrow (\theta,-\rho),$$
$$\gamma_2: (\theta,\rho) \rightarrow (1/\theta,\rho/\theta^4).$$
It is easily seen that $\gamma_1,\gamma_2 \in Aut(\mathcal C_{a,3})$ and that $G_3 = \langle \gamma_1,\gamma_2 \rangle$ is elementary abelian of order $4$. Since $\gamma_1$ is the hyperelliptic involution of $\mathcal C_{a,c}$, both $Fix\langle \gamma_1 \rangle={\overline{\mathbb F}_p}(\theta)$ and $FixG_3$ are rational.
We first show that $J_{a,3}$ decomposes as a the product of the Jacobians of an elliptic curve and of a curve of genus $2$. 

\begin{lemma} \label{lemma6}
Let $q=p^h$ where $h \in \mathbb{N}$ and $p$ is an odd prime. Let $a \in \mathbb{F}_{q^2}$ with $a^2+108 \ne 0$. Then $Fix\langle\gamma_2\rangle={\overline{\mathbb F}_p}(\delta,\nu)$ where
\begin{equation} \label{quot32}
\nu^2=\delta^3-a\delta^2-36\delta+4a.
\end{equation}
In particular $Fix\langle \gamma_2\rangle$ is elliptic.
\end{lemma}

\begin{proof}
Let $\delta=(\theta^2+1)/\theta=\theta+1/\theta=\theta+\gamma_2(\theta)$ and $\nu=-\rho/\theta^2$.
Then $\delta$ and $\rho$ are fixed by $\gamma_2$ since $\gamma_2(\delta)=\gamma_2(\theta)+\gamma_2^2(\theta)=\gamma_2(\theta)+\theta$ and $\gamma_2(\nu)=(-\rho/\theta^4)\cdot \theta^2=-\rho/\theta^2=\nu$. Also from \eqref{quot2},
\begin{eqnarray*}
\nu^2&=&\frac{\rho^2}{\theta^4}=\frac{\theta^7-a\theta^6-33\theta^5+2a\theta^4-33\theta^3-a\theta^2+\theta}{\theta^4}\\
&=&\frac{\theta^6-a\theta^5-33\theta^4+2a\theta^3-33\theta^2-a\theta+1}{\theta^3},
\end{eqnarray*}
whereas
\begin{eqnarray*}
\delta^3-a\delta^2-36\delta+4a&=&\frac{(\theta^2+1)^3-a\theta(\theta^2+1)^2-36\theta^2(\theta^2+1)+4a\theta^3}{\theta^3}\\
&=&\frac{\theta^6-a\theta^5-33\theta^4+2a\theta^3-33\theta^2-a\theta+1}{\theta^3}.
\end{eqnarray*}
Then the rest of the proof is analogous to that of Lemma \ref{lemma4}.
\end{proof}

In the following remark we show that the quotient curve $\mathcal{C}_{a,3}/\langle \gamma_2 \rangle$ and $ E_1$ are $\mathbb{F}_{q^2}$-isomorphic.

\begin{remark} \label{isomelliptic}
Let $q=p^h$ where $h \in \mathbb{N}$ and $p$ is an odd prime. Let $a \in \mathbb{F}_{q^2}$ with $a^2+108 \ne 0$. Then the elliptic curves $ E_1: Y^2=4X^3+(a+6)X^2+(a-6)X-4$ and $\mathcal{C}_{a,3}/\langle \gamma_2 \rangle: Y^2=X^3-aX^2-36X+4a$ are $\mathbb{F}_{q^2}$-isomorphic through the birational morphism $\varphi:  E_1 \rightarrow \mathcal{C}_{a,3}/\langle \gamma_2 \rangle$ given by
$$\varphi(X,Y) \mapsto (-4X-2,4\xi Y),$$
where $\xi^2=-1$.
\end{remark}

\begin{lemma} \label{lemma7}
Let $q=p^h$ where $h \in \mathbb{N}$ and $p$ is an odd prime. Let $a \in \mathbb{F}_{q^2}$ with $a^2+108 \ne 0$. Let $\gamma_3=\gamma_1 \gamma_2$ so that $\gamma_3: (\theta,\rho) \mapsto (1/\theta, -\rho/\theta^4)$. Then $Fix\langle \gamma_3\rangle={\overline{\mathbb F}_p}(\delta,\epsilon)$ where
\begin{equation} \label{quot33}
\epsilon^2=\delta^5-a\delta^4-40\delta^3+8a\delta^2+144\delta-16a=(\delta^2-4)(\delta^3-a\delta^2-36\delta+4a).
\end{equation}
In particular, $Fix\langle \gamma_3\rangle$ is hyperelliptic of genus $2$.
\end{lemma}

\begin{proof}
Let $\delta=(\theta^2+1)/\theta=\theta+\/\theta=\theta+\gamma_3(\theta)$ and $\epsilon=(\theta^2\rho-\rho)/\theta^3$. Then $\delta$ and $\epsilon$ are fixed by $\gamma_3$ since $\gamma_3(\delta)=\gamma_3(\theta)+\gamma_3^2(\theta)=\gamma_3(\theta)+\theta$ and $\gamma_3(\epsilon)=(-\rho/\theta^6+\rho/\theta^4) \cdot \theta^3=(-\rho+\rho\theta^2)/\theta^3=\epsilon$.
Using Equation \eqref{quot2} one can easily check that  $\epsilon^2=\delta^5-a\delta^4-40\delta^3+8a\delta^2+144\delta-16a$ holds. Arguing as in the previous proofs, we have that ${\overline{\mathbb F}_p}(\delta,\epsilon)$ coincides with $Fix\langle \gamma_3\rangle $ and it is hyperelliptic of genus $2$. 
\end{proof}

The following corollary summarizes the results obtained in this section so far.

\begin{corollary} \label{KR3}
Let $q=p^h$ where $h \in \mathbb{N}$ and $p$ is an odd prime. Let $a \in \mathbb{F}_{q^2}$ with $a^2+108 \ne 0$. Then the Jacobian variety $J_{a,3}$ of the curve $\mathcal{C}_{a,3}$ defined in \eqref{gen3} decomposes as $J_{a,3} \sim J_{a,3,2} \times J_{ E_1}$  where $J_{ E_1}$ is the Jacobian variety of the elliptic curve
\begin{equation} \label{primaellittica2}
{ E_1}: Y^2=4X^3+(a+6)X^2+(a-6)X-4,
\end{equation}
and $J_{a,3,2}$ is the Jacobian variety of the genus $2$ curve
\begin{equation} \label{ultimagen2}
\mathcal{C}_{a,3,2}: Y^2=X^5-aX^4-40X^3+8aX^2+144X-16a.
\end{equation}
Furthermore, the Jacobian variety $J_a$ of the curve $\mathcal{C}_a$ defined in \eqref{family} decomposes as
\begin{equation} \label{decomposition3}
{J_a} \sim J_{ E_1}^3 \times {J_{a,3,2}}, 
\end{equation}
\end{corollary}

In the rest of the section we investigate the curve $\mathcal{C}_{a,3,2}$, whose function field is $\overline{\mathbb F}_{q}(\delta,\epsilon)$, where Equation \eqref{quot33} holds. 

For $k\in \mathbb F_{p^2}$, $k\neq \pm 2 $ and $k\neq \pm 6$, let $\lambda=\frac{12-k^2}{2k}$ and 
$a=\frac{\lambda^3-36\lambda}{\lambda^2-4}$. 
%
Then, by straightforward computation
$$
\epsilon^2=(\delta^2-4)(\delta-\lambda) \left(\delta +\frac{2\lambda+12}{\lambda-2}\right)\left(\delta +\frac{-2\lambda+12}{\lambda+2}\right).
$$
We consider the group generated by the hyperelliptic involution and 
$$\psi_i(\delta,\epsilon)=\left(\frac{\lambda \delta+12}{\delta-\lambda},\frac{(-1)^i\mu \epsilon}{(\delta-\lambda)^3}\right),$$
$i=0,1$, where
$\mu^2=(\lambda^2+12)^3$,
which is elementary abelian of order $4$. Note that 
$\mu \in \mathbb{F}_{p^2}$, since $\lambda^2+12=\left(\frac{k^2+12}{2k}\right)^2$ is a square in $\mathbb{F}_{p^2}$.

The functions
$$\frac{(\delta-\lambda)(\lambda^2+12)+(-1)^i\mu}{(\delta-\lambda)^2(\lambda^2+12)}\cdot \epsilon, \qquad \frac{\delta^2+12}{\delta-\lambda}$$
are fixed by $\psi_i$.

By direct computations, the quotient curves $\mathcal{C}_{a,3,2}^i=\mathcal{C}_{a,3,2}/\langle\psi_i \rangle$ have equations

\begin{eqnarray*}
\left(\frac{1}{4}  k^4 - 10  k^2 + 36\right) \epsilon^2&=&    \left(\frac{1}{4}  k^4 - 10  k^2 + 36\right) \delta^3 + \left(\frac{1}{2}  k^5 - 36  k^3+ 264  k\right)\delta^2 \\
&& +(- 32  k^4 + 640  k^2 )\delta + 512 k^3;\\
\left(\frac{1}{4}  k^5 - 10  k^3 + 36  k\right)\epsilon^2 &=& \left(\frac{1}{4}  k^5 - 10  k^3 + 36 k\right)\delta^3 +(- 22 k^4 + 432  k^2 - 864 )\delta^2\\
&&+ (640  k^3 - 4608  k)\delta  - 6144 k^2.\\
\end{eqnarray*}
It turns out that the two elliptic curves are isogenous over $\mathbb{F}_{p^2}$ via  the isogeny $\varphi: \mathcal{C}_{a,3,2}^0 \rightarrow \mathcal{C}_{a,3,2}^1$ with

$$\varphi: (\delta,\epsilon) \mapsto \left(  \frac{a_3 \delta^3+a_2 \delta^2+a_1 \delta+a_0}{(\delta+c)^2}, \frac{\delta\epsilon(b_2 \delta^2+b_1 \delta+b_0)}{(\delta+c)^3}\right),$$

where 
\begin{eqnarray*}
a_0 = -\frac{1536k}{(k^2 - 36))^2k}, \quad 
a_1 = -\frac{3072k}{(k^2 - 36)k}, \quad
a_2 = \frac{32}{k}, \quad 
a_3 = \frac{4}{k^2},\\
b_0 = -\frac{2048}{k^3 - 36 k}, \quad 
b_1 = \frac{1536k^2}{k^2 - 36}, \quad 
b_2 = \frac{8}{k^3}, \quad 
c = -\frac{64k}{k^2 - 36}.\\
\end{eqnarray*}

Now we need to investigate whether $\mathcal{C}_{a,3,2}^0$ and $ E_1$ are isogenous or not. To this end, we will consider the theory of modular polynomials, for which we refer to \cite{Modular1} and the web database \cite{Modular2}. We proceed by arguing as in  \cite[Section 1]{south}.

Let $\Phi_\ell \in \mathbb{Z}[X,Y]$ be the classical modular polynomial of degree $\ell$. It parametrizes pairs of $\ell$-isogenous elliptic curves in terms of their $j$-invariants. More precisely, over a field $\mathbb{F}$ of characteristic not equal to $\ell$, the modular equation $\Phi_\ell(j_1,j_2)= 0$ holds if and only if $j_1$ and $j_2$ are the $j$-invariants of elliptic curves defined over $\mathbb{F}$ that are related by a cyclic isogeny of degree $\ell$. 

Here, $\mathbb{F}$ is a finite field $\mathbb{F}_{p^2}$.
Let $\overline{j}(k)$, and $\widetilde{j}(k)$ be the  $j$-invariants  of $\mathcal{C}_{a,3,2}^0$ and $ E_1$, respectively, both expressed as a function of $k$. We have
$$
\overline{j}(k)=\frac{(k^2+12)^3(k^6-60k^4+1200k^2+192)^3}{64(k^2-36)^2(k^2-4)^6k^2}, \qquad
\widetilde{j}(k)=\frac{4(k^2+12)^6}{(k^2-36)^2(k^2-4)^2k^2}.
$$


We consider the case $\ell=3$.
 The class of  elliptic curves $3$-isogenous to $E_1$ can be determined by computing (if any) the roots of $\Phi_3(Y) = \Phi_3(j(E_1),Y) \in \mathbb{F}_{p^2}[Y]$ over $\mathbb{F}_{p^2}$. 
Hence we search for  $\overline{k} \in \mathbb{F}_{p^2}$ with $\Phi_{3}(\overline{j}(\overline{k}), \widetilde{j}(\overline{k}) )=0$ satisfying Condition \eqref{Condizione}. 

We are not  able a priori to decide whether the $3$-isogeny is between $E_1$ and $C_{a,3,2}^0$ or its twists. So, the condition $\Phi_{3}(\overline{j}(\overline{k}), \widetilde{j}(\overline{k}) )=0$ ensures the existence of  a $3$-isogeny over $\mathbb{F}_{p^4}$ between  $\mathcal{C}_{a,3,2}^0$ and $E_1$. 
Thus we look for explicit descriptions of such $3$-isogenies and check if they are defined over $\mathbb{F}_{p^2}$.

\begin{itemize}
    \item $\overline{k}^2 = 36/5$. In this case the $3$-isogeny between $\mathcal{C}_{a,3,2}^0$ and   $ E_1$ is $(x,y) \mapsto \left(\frac{\psi_x}{\theta_x},\frac{\psi_y}{\theta_y}\right)$, 
    where
    \begin{eqnarray*}
    \psi_x&=&-360 x^3 \overline{k} + 2160 x^3 + 14688 x^2 \overline{k} - 36288 x^2 - 138240 x \overline{k} + 331776 x + 368640 \overline{k} - 1105920;\\
    \theta_x &=&360 x^3 \overline{k} - 1296 x^3 - 10368 x^2 - 69120 x \overline{k} + 82944 x + 92160 \overline{k} - 552960;\\
    \psi_y &=& y(103680 x^4  \overline{k} - 124416 x^4  - 1105920 x^3  \overline{k} + 6635520 x^3  + 17694720 x^2 y \overline{k}\\ &&- 21233664 x^2  - 17694720 x  \overline{k} +     106168320 x);\\
    \theta_y &=&-6480 x^6 \overline{k} + 12960 x^6 - 103680 x^5 \overline{k} - 103680 x^4 \overline{k} - 6842880 x^4 - 15482880 x^3 \overline{k}\\&& - 53084160 x^3 - 225607680 x^2 \overline{k} -     79626240 x^2 - 2548039680 x - 1415577600 \overline{k}.
    \end{eqnarray*}

     \item $\overline{k}^2 = -4/7$. In this case the $3$-isogeny between $\mathcal{C}_{a,3,2}^0$ and   $ E_1$ is $(x,y) \mapsto \left(\frac{\psi_x}{\theta_x},\frac{\psi_y}{\theta_y}\right)$, where
     {\small
    \begin{eqnarray*}
    \psi_x&=&9756997844651904 x^3 \overline{k} + 1773999608118528 x^3 - 236713557425574852 x^2 \overline{k}\\
    &&- 88881615135103896 x^2 -    149484312145104132 x \overline{k} + 93592177008289512 x\\
    &&+ 108081789310424734 \overline{k} + 80629249554918068;\\
    \theta_x &=&-9756997844651904 x^3 \overline{k} + 14445425380393728 x^3 + 117767323565720640 x^2 \overline{k}\\
    &&+ 35210207348410752 x^2 +   76832921522684208 x \overline{k} - 181553980415546592 x\\
    &&- 164874109067025579 \overline{k} - 34657310309390402;\\
    \psi_y &=& 5y(3457479429497341093236434568000000 x^4 \overline{k} - 1014193965985886720682687473280000 x^4 \\
    &&-    15545163634077077821994778380160000 x^3 \overline{k} - 2826393388014014149453596069120000 x^3\\
    &&-   117406724988018917190005308164803520 x^2 \overline{k} + 35214070050544757907472923900067776 x^2\\
    &&+   897502873974695889338633745055844520 x \overline{k} + 43112385947023817822534160466259664 x\\
    &&-  86384192632333382597513730705421100 \overline{k} - 65824077521259664661193935638615080)/6;\\
    \theta_y &=&12012271389339333626787155527680000 x^6 \overline{k} + 4636315273078339294549428449280000 x^6\\ 
    &&+  58558963160486989619604042762240000 x^5 \overline{k} - 111920992422190876567021530009600000 x^5\\
    &&-  753377091175625014779032837560320000 x^4 \overline{k} - 171585897605402253148919969648640000 x^4\\
    &&-  455341963965646265965574769587040000 x^3 \overline{k} + 1538861060386352212314316012831680000 x^3\\
    &&+ 3093062942957171752624793406954480000 x^2 \overline{k} + 374348928145257960497936027483040000 x^2\\
    &&+ 274459963096197145160822714714460000 x \overline{k} - 1900058781725511243895076291080920000 x\\
    &&-  855848892852162349927591124235841875 \overline{k} - 45599538718475703951502910235251250.
    \end{eqnarray*}}
    \item $\overline{k}^2 = 24\overline{k}+36$. In this case the $3$-isogeny between $\mathcal{C}_{a,32}^0$ and  $ E_1$ is $(x,y) \mapsto \left(A(x),cyA^{\prime}(x)\right)$, 
   $A(x)=\frac{a_3x^3+a_2x^2+a_1x+a_0}{x^3+b_2x^2+b_1x+b_0}$, $c^2 - 1/3c - 1/9=0$, and 
    \begin{eqnarray*}
    \overline {k} a_3 + 2/7 \overline{k}  + 6/7 a_3&=&0\\    
    a_2^2 + 276/5 a_2 + 17424/25& =&0\\
    a_1^2 - 384 a_1  + 36864/5 &=&0\\
     a_0^2 + 512 a_0  + 65536/5 &=&0\\
     b_2^2 - 12 b_2  - 1584/5 &=&0\\
     b_1^2 - 192 b_1  + 36864/5 &=&0\\
     b_0^2 + 256 b_0  + 65536/5&=&0.\\
    \end{eqnarray*}
    
     \item $\overline{k}^2 = -24\overline{k}+36$. In this case the $3$-isogeny between $\mathcal{C}_{a,32}^0$ and  $ E_1$ is $(x,y) \mapsto \left(A(x),cyA^{\prime}(x)\right)$, 
   $A(x)=\frac{a_3x^3+a_2x^2+a_1x+a_0}{x^3+b_2x^2}$, $c^2 - 1/3c - 1/9=0$, and 
    \begin{eqnarray*}
    \overline {k} a_3 - 2/5 \overline{k}  - 6/5 a_3&=&0\\    
    a_2^2 - 132a_2 + 11376/5& =&0\\
   a_1^2 - 768a_1 + 589824/5&=&0\\
     a_0^2 - 1024a_0 + 1048576/5&=&0\\
     b_2^2 + 48 b_2 - 144&=&0.\\
    \end{eqnarray*}
\end{itemize}

We summarize all the results about the curve $\mathcal{C}_a$ in the following theorem.

\begin{theorem}\label{Genus5:mainTheorem}
Let $q=p^h$, where $h\in \mathbb{N}$ and $p$ is an odd prime.  Let 
$$a=\frac{-k^6 + 180 k^4 - 2160 k^2 + 1728}{2 k^5 - 80 k^3 + 288 k}\in \mathbb{F}_{q^2}, \qquad k \in \mathbb{F}_{q^2}, \qquad a^2+ 108\neq 0.$$ 

If $k^2=36/5$, $k^2=-4/7$, or $k^2=\pm24 k^2+36$, then  $a \in \mathbb{F}_{p^2}$ and the Jacobian variety $\mathcal{J}_a$ of the curve $\mathcal{C}_a$ defined in \eqref{family}  is $\mathbb{F}_{p^2}$-isogenous to the fifth power of the Jacobian variety $J_{ E_1}$ of the elliptic curve 
\begin{equation*}
{ E_1}: Y^2=4X^3+(a+6)X^2+(a-6)X-4.
\end{equation*}
In particular, there exist infinitely many primes $p$ such that $\mathcal{C}_a$ is either maximal or minimal over $\mathbb{F}_{p^2}$.
\end{theorem}
\proof
We only need to prove that there exists a fixed elliptic curve $E$ defined over $\mathbb{Q}$ such that $E_1$ and $E$ are $\mathbb{F}_{p^4}$-isogenous for any $p$. This is clear since the $j$-invariant of $E_1$ is a rational number. 
\endproof

Using the software MAGMA \cite{MAGMA} we determined all the primes $p$ less than $100000$ such that one of the above conditions on $\overline{k}\in \mathbb{F}_p$ is satisfied and $ E_1$ is $\mathbb{F}_{p^2}$-maximal. 

\begin{theorem}\label{Th:Massimalita}
For $a$ and $p$ as in Tables \ref{Table1}, \ref{Table2}, \ref{Table3}, \ref{Table4}, the curve $\mathcal{C}_{a}$ defined in \eqref{gen3} is $\mathbb{F}_{p^2}$-maximal.
\end{theorem}

\begin{table}[ht]
    \centering
    \begin{tabular}{|l|l||l|l||l|l||l|l||l|l|}
    \hline
    $p$&$a$&$p$&$a$&$p$&$a$&$p$&$a$&$p$&$a$\\
    \hline
$  11  $   &   $  0  $   & 
 $  131  $   &   $  75  $   & 
 $  251  $   &   $  221  $   & 
 $  491  $   &   $  254  $   & 
 $  599  $   &   $  24  $   \\ 
 $  1439  $   &   $  984  $   & 
 $  3371  $   &   $  933  $   & 
 $  5639  $   &   $  4445  $   & 
 $  5879  $   &   $  1294  $   & 
 $  6971  $   &   $  2129  $   \\
 $  7079  $   &   $  6366  $   & 
 $  8039  $   &   $  4751  $   & 
 $  8291  $   &   $  5367  $   & 
 $  9839  $   &   $  7465  $   & 
 $  10799  $   &   $  1835  $   \\ 
 $  11171  $   &   $  5951  $   & 
 $  12119  $   &   $  4310  $   & 
 $  14879  $   &   $  2439  $   & 
 $  16931  $   &   $  14808  $   & 
 $  17159  $   &   $  4688  $   \\
 $  18839  $   &   $  17890  $   & 
 $  23039  $   &   $  5286  $   & 
 $  23159  $   &   $  16126  $   & 
 $  25919  $   &   $  24335  $   & 
 $  50291  $   &   $  5355  $   \\
 $  53411  $   &   $  49474  $   & 
 $  53639  $   &   $  10108  $   & 
 $  59051  $   &   $  25275  $   & 
 $  69371  $   &   $  51864  $   & 
 $  74771  $   &   $  11482  $   \\ 
 $  74891  $   &   $  48023  $   & 
 $  75239  $   &   $  47655  $   & 
 $  81119  $   &   $  7710  $   & 
 $  81359  $   &   $  48587  $   & 
&\\
\hline
    \end{tabular}
    \caption{Values $a\in \mathbb{F}_p$ for which $\mathcal{C}_{a,3}$ defined in \eqref{gen3} is $\mathbb{F}_{p^2}$-maximal ($\overline{k}^2=36/5$)}
    \label{Table1}
\end{table}

\begin{table}[ht]
    \centering
    \begin{tabular}{|l|l||l|l||l|l||l|l||l|l|}
    \hline
    $p$&$a$&$p$&$a$&$p$&$a$&$p$&$a$&$p$&$a$\\
    \hline
$11$& $0$&
$23$& $0$&
$71$& $45$&
$263$& $185$&
$1031$& $881$\\
$1283$& $1276$&
$1583$& $207$&
$2039$& $1586$&
$2087$& $752$&
$2543$& $565$\\
$2843$& $2158$&
$7823$& $2710$&
$9851$& $6895$&
$10859$& $7298$&
$12107$& $1665$\\
$15647$& $4803$&
$18719$& $9865$&
$23459$& $17387$&
$25463$& $14115$&
$26723$& $22923$\\
$27791$& $27198$&
$29759$& $25238$&
$31511$& $30573$&
$32579$& $25335$&
$33791$& $31173$\\
$34283$& $3491$&
$35099$& $20799$&
$37619$& $25640$&
$54959$& $54405$&
$56519$& $47622$\\
$61007$& $47142$&
$61559$& $42225$&
$67043$& $32831$&
$71339$& $1126$&
$71843$& $10626$\\
$89783$& $68842$&
$95027$& $9806$&
&&
&&
&\\
\hline
    \end{tabular}
    \caption{Values $a\in \mathbb{F}_p$ for which $\mathcal{C}_{a,3}$ defined in \eqref{gen3} is $\mathbb{F}_{p^2}$-maximal ($\overline{k}^2=-4/7$)}
    \label{Table2}
\end{table}

\begin{table}[ht]
    \centering
    \begin{tabular}{|l|l||l|l||l|l||l|l||l|l|}
    \hline
    $p$&$a$&$p$&$a$&$p$&$a$&$p$&$a$&$p$&$a$\\
    \hline
 $  11  $   &   $  0  $   & 
 $  131  $   &   $  56  $   & 
 $  251  $   &   $  221  $   & 
 $  491  $   &   $  237  $   & 
 $  599  $   &   $  575  $   \\
 $  1439  $   &   $  455  $   & 
 $  3371  $   &   $  2438  $   & 
 $  5639  $   &   $  1194  $   & 
 $  5879  $   &   $  1294  $   & 
 $  6971  $   &   $  2129  $   \\ 
 $  7079  $   &   $  713  $   & 
 $  8039  $   &   $  3288  $   & 
 $  8291  $   &   $  2924  $   & 
 $  9839  $   &   $  7465  $   & 
 $  10799  $   &   $  1835  $   \\ 
 $  11171  $   &   $  5951  $   & 
 $  12119  $   &   $  4310  $   & 
 $  14879  $   &   $  12440  $   & 
 $  16931  $   &   $  2123  $   & 
 $  17159  $   &   $  12471  $   \\ 
 $  18839  $   &   $  17890  $   & 
 $  23039  $   &   $  5286  $   & 
 $  23159  $   &   $  7033  $   & 
 $  25919  $   &   $  24335  $   & 
 $  50291  $   &   $  5355  $   \\ 
 $  53411  $   &   $  3937  $   & 
 $  53639  $   &   $  10108  $   & 
 $  59051  $   &   $  33776  $   & 
 $  69371  $   &   $  51864  $   & 
 $  74771  $   &   $  63289  $   \\ 
 $  74891  $   &   $  48023  $   & 
 $  75239  $   &   $  27584  $   & 
 $  81119  $   &   $  7710  $   & 
 $  81359  $   &   $  48587  $   & 
 &\\
\hline
    \end{tabular}
    \caption{Values $a\in \mathbb{F}_p$ for which $\mathcal{C}_{a,3}$ defined in \eqref{gen3} is $\mathbb{F}_{p^2}$-maximal ($\overline{k}^2=24\overline{k}+36$)}
    \label{Table3}
\end{table}

\begin{table}[ht]
    \centering
    \begin{tabular}{|l|l||l|l||l|l||l|l||l|l|}
    \hline
    $p$&$a$&$p$&$a$&$p$&$a$&$p$&$a$&$p$&$a$\\
    \hline
 $  11  $   &   $  0  $   & 
 $  131  $   &   $  56  $   & 
 $  251  $   &   $  30  $   & 
 $  491  $   &   $  237  $   & 
 $  599  $   &   $  24  $   \\
 $  1439  $   &   $  455  $   & 
 $  3371  $   &   $  933  $   & 
 $  5639  $   &   $  4445  $   & 
 $  5879  $   &   $  4585  $   & 
 $  6971  $   &   $  4842  $   \\ 
 $  7079  $   &   $  6366  $   & 
 $  8039  $   &   $  3288  $   & 
 $  8291  $   &   $  2924  $   & 
 $  9839  $   &   $  2374  $   & 
 $  10799  $   &   $  1835  $   \\ 
 $  11171  $   &   $  5220  $   & 
 $  12119  $   &   $  4310  $   & 
 $  14879  $   &   $  2439  $   & 
 $  16931  $   &   $  14808  $   & 
 $  17159  $   &   $  12471  $   \\ 
 $  18839  $   &   $  17890  $   & 
 $  23039  $   &   $  5286  $   & 
 $  23159  $   &   $  16126  $   & 
 $  25919  $   &   $  1584  $   & 
 $  50291  $   &   $  5355  $   \\
 $  53411  $   &   $  3937  $   & 
 $  53639  $   &   $  10108  $   & 
 $  59051  $   &   $  33776  $   & 
 $  69371  $   &   $  17507  $   & 
 $  74771  $   &   $  11482  $   \\ 
 $  74891  $   &   $  48023  $   & 
 $  75239  $   &   $  27584  $   & 
 $  81119  $   &   $  73409  $   & 
 $  81359  $   &   $  32772  $   & 
 &\\
\hline
    \end{tabular}
    \caption{Values $a\in \mathbb{F}_p$ for which $\mathcal{C}_{a,3}$ defined in \eqref{gen3} is $\mathbb{F}_{p^2}$-maximal ($\overline{k}^2=-24\overline{k}+36$)}
    \label{Table4}
\end{table}

\section{A characterization in terms of automorphism groups}\label{Section:Genus5Characterization}

In this section we extend to the positive characteristic case a characterization of curves with equation \eqref{family} in terms of their automorphism group, provided by Shaska in \cite{shas}.

Let $q=p^h$ where $h \in \mathbb{N}$ and $p$ is an odd prime. Let $a \in \mathbb{F}_{q^2}$ with $a^2+108 \ne 0$. Denote by ${\overline{\mathbb F}_p}$ the algebraic closure of $\mathbb{F}_{q^2}$. We consider the  family of algebraic curves $\mathcal C_a$
%
as  described in \eqref{family}.

We know that $g(\mathcal{C}_a)=5$, $\mathcal{C}_a$ is hyperelliptic and that the hyperelliptic involution is
$$\alpha_1:(X,Y) \mapsto (X,-Y).$$
Our first aim is to show that $Aut(\mathcal{C}_a)$ admits a subgroup $H$ with $H \cong \rm{A_4} \times C_2 \cong SmallGroup(24,13)$ where $\rm{A}_4$ denotes the alternating group in $4$ letters. Then we prove that if $\mathcal{X}_5$ is an hyperelliptic curve of genus $5$ admitting an automorphism group isomorphic to $H$ then $\mathcal{X}$ is ${\overline{\mathbb F}_p}$-birationally equivalent to a curve $\mathcal{C}_a$ in \eqref{family}.


Let $i \in \mathbb{F}_{q^2}$ with $i^2=-1$. Consider the following map
$$\beta:(X,Y) \mapsto \bigg( \frac{X-i}{X+i}, \frac{8iY}{(X+i)^6}\bigg).$$
It is easily seen that  $\beta \in Aut(\mathcal{C}_a)$. 
Also, $\beta$ has order $3$.
Involutory automorphisms of $\mathcal C_a$ are described as follows:
$$\alpha_2:(X,Y) \mapsto (-X,Y) \quad \textrm{and} \quad \alpha_3:(X,Y) \mapsto \bigg( \frac{1}{X}, \frac{Y}{X^6}\bigg).$$
The group generated by $\alpha_1$, $\alpha_2$, and $\alpha_3$ is an elementary abelian group of order $8$.

To show that $H_1=\langle \alpha_2,\beta \rangle \cong \rm{A}_4$ and that $\alpha_1 \not\in H_1$,  
 we first observe that
$$\beta^{-1} \alpha_2 \beta=\alpha_3.$$
Then  $\beta$ normalizes the elementary abelian group of order $4$ generated by $\alpha_2$ and $\alpha_3$.
%
%
This implies that $H_1 \cong (C_2 \times C_2) \rtimes C_3 \cong \rm{A}_4$. Also,  $\alpha_1 \not\in \langle \alpha_2, \alpha_3 \rangle$ yields $\alpha_1 \not\in H_1$. This proves that $Aut(\mathcal{C}_a)$ contains a subgroup $H=\langle H_1,\alpha_1 \rangle =H_1 \times \langle \alpha_1 \rangle \cong \rm{A}_4 \times C_2 \cong SmallGroup(24,13)$.

Assume now that, for an odd prime $p$,  $\mathcal{X}_5$ is a hyperelliptic curve of genus $5$ defined over a finite field of characteristic $p$ admitting an automorphism group $H \cong \rm{A}_4 \times C_2$. Our aim is to prove that $\mathcal{X}_5$ is ${\overline{\mathbb F}_p}$-birationally equivalent to a curve belonging to the family in \eqref{family}.

The curve $\mathcal{X}_5$ admits an affine model
$$\mathcal{X}_5:Y^2=F(X)=\prod_{i=1}^{12} (X-\gamma_i),$$
where the points $P_i=(\gamma_i,0)$, $i=1,\ldots,12$, are the Weierstrass points of $\mathcal{X}_5$; see \cite[Proposition 6.2.3]{Sti} and \cite[Theorem 11.98]{HKT}. Let ${\overline{\mathbb F}_p}(x,y)$ be the function field of $\mathcal{X}_5$. The hyperelliptic involution $\alpha_1 \in H$ is $\alpha_1:(X,Y) \mapsto (X,-Y)$ and its fixed field is ${\overline{\mathbb F}_p}(x)$. The fixed points of $\alpha_1$ on $\mathcal{X}_5$ are exactly $P_i$,  $i=1,\ldots,12$.


Let $\overline {H}=H/\langle \alpha_1 \rangle \cong \rm{A}_4 \leq \textrm{PGL}(2,{\overline{\mathbb F}_p})$.
From \cite[Theorem 1]{VM} (see also \cite[Theorem A.8]{HKT}), any subgroup of $\textrm{PGL}(2,{\overline{\mathbb F}_p})$ which is isomorphic to $\rm{A}_4$ is conjugated to $\overline{ H}$ in $\textrm{PGL}(2,{\overline{\mathbb F}_p})$. This means that up to conjugation we can assume that $\overline{ H}$, seen as an automorphism group of the rational function field ${\overline{\mathbb F}_p}(x)$, is generated by the following maps
$$\theta_1:x \mapsto -x, \quad and \quad \theta_2:x \mapsto \frac{x-i}{x+i},$$
where $i^2=-1$. 

\begin{lemma} \label{trans}
The subgroup $H_1$ of $H$ with $H_1 \cong \rm{A}_4$ acts sharply transitively on the set of Weierstrass points of $\mathcal{X}_5$.
\end{lemma}

\begin{proof}
Since $H_1$ is an automorphism group of $\mathcal{X}_5$, it acts on the set $W$ of its Weierstrass points. Since the stabilizer in $Aut(\mathcal{X}_5)$ of a point $P \in W$ is cyclic from \cite[Lemma 11.44]{HKT} and clearly contains $\alpha_1$ 
we deduce that the $3$ involutions in $H_1$ fix no points on $W$. 
Suppose that  $H_1$ does not act transitively on $W$. Then the stabilizer $H_{1,P}$ of $P$ has order $3$. A nontrivial element $\delta \in H_{1,P}$ fixes at least three points of $W$, since the number of fixed points must be congruent to $0$ modulo $3$. Consider the element $\overline{\delta}$ induced by $\delta$ on $Aut(\mathcal{X}_5/\langle \alpha_1 \rangle)$. Each point $\overline P$ lying under a fixed point of $\delta$ in $W$ is fixed by $\overline{\delta}$ 
On the other hand, by \cite[Theorem 11.14 (d)]{HKT}, $\overline{\delta}$ fixes exactly $2$ points on $\mathcal{X}_5/\langle \alpha_1 \rangle$, a contradiction.
%
\end{proof}

\begin{theorem}\label{caratterizzazione}
Let $p$ be an odd prime and let ${\overline{\mathbb F}_p}$ be the algebraic closure of $\mathbb F_p$. Let $\mathcal{X}_5 \ : \ Y^2=F(X)$ be an hyperelliptic curve of genus $5$ defined over ${\overline{\mathbb F}_p}$  and let $\alpha_1$ denote its hyperelliptic involution. If $\mathcal{X}_5$ admits an automorphism group $H\cong \rm{A}_4 \times C_2$ with $\overline{H}=H/\langle \alpha_1 \rangle \cong \rm{A}_4$ then $\mathcal{X}_5$ is ${\overline{\mathbb F}_p}$-birationally equivalent to a curve $\mathcal{C}_a$ given in \eqref{family}.
 If in addition $\mathcal{X}_5$ is defined over a finite field $\mathbb{F}_q$ with $q \equiv 1 \pmod 4$ and  admits at least an $\mathbb{F}_q$-rational Weierstrass point $P=(t,0)$,  then $\mathcal{X}_5$ is $\mathbb F_q$-birationally equiavalent to $\mathcal{C}_a$. \end{theorem}
\proof
Without loss of generality assume that $\deg(F(x))=12$.
Lemma \ref{trans} shows that the roots of $F(X)$ form an orbit under the action of an automorphism group of ${\overline{\mathbb F}_p}(x)$ isomorphic to $A_4$, namely $\overline H=H/\langle \alpha_1 \rangle$. Then there is an automorphism $\gamma$ of ${\overline{\mathbb F}_p}(x)$ such that
$$
\gamma\overline H\gamma^{-1}=\langle \theta_1, \theta_2 \rangle.
$$
Hence $\gamma$ maps the roots of $F(X)$ to an orbit $\Omega$ under the action of $\langle \theta_1, \theta_2 \rangle$. If $t$ is a point in $\Omega$, then $\Omega$ consists of

$$t, \ \theta_1(t)=-t, \ \theta_2(t)=\frac{t-i}{t+i}, \ \theta_2^2(t)=-i\frac{t+1}{t-1}, \  \theta_1\theta_2(t)=-\frac{t-i}{t+i}, \  \theta_2\theta_1(t)=\frac{t+i}{t-i}, \ \theta_1\theta_2\theta_1(t)=-\frac{t+i}{t-i},$$
$$ \theta_2^2\theta_1(t)=-i\frac{t-1}{t+1}, \ \theta_1\theta_2^2 (t)=i\frac{t+1}{t-1},\ 
\theta_1\theta_2^2 \theta_1(t)=i\frac{t-1}{t+1}, \ \theta_2^2\theta_1\theta_2(t)=-\frac{1}{t}, \ \theta_1\theta_2^2\theta_1\theta_2(t)=\frac{1}{t}.$$
Label these values with $z_1,\ldots,z_{12}$.  Then
$G(X):=\prod_{i=1}^{12}(X-z_i)$ coincides with
\begin{eqnarray*}
X^{12} + \frac{-t^{12} + 33t^8 + 33t^4 - 1}{t^{10} - 2t^6 + t^2}x^{10} - 33X^8  
+2\frac{t^{12} - 33t^8 - 33t^4 + 1}{t^{10} - 2t^6 + t^2}X^6\\- 33X^4+ \frac{-t^{12} +    33t^8 + 33t^4 - 1}{t^{10} - 2t^6 + t^2}X^2 + 1,
\end{eqnarray*}
Since the roots of $F(X)$ are mapped to the roots of $G(X)$ by an automorphism of $\overline{\mathbb F}_p(x)$, the curve $\mathcal X_5$ is birationally equivalent to that of equation $Y^2=G(X)$, which clearly belongs to the family described in \eqref{family} for
$$a=\frac{t^{12} - 33t^8 - 33t^4 + 1}{t^{10} - 2t^6 + t^2}=\frac{t^{12} - 33t^8 - 33t^4 + 1}{t^2(t^4-1)^2}.$$

Note that if $q\equiv 1 \pmod{4}$ then the group $H$ is defined over  $\mathbb{F}_q$ 
and  $\overline{H}$ is conjugated to $\langle \theta_1,\theta_2\rangle$ in $\textrm{PGL}(2,q)$; see \cite[Theorem A.8]{HKT}. This implies that if the curve $\mathcal{X}_5$ admits a Weierstrass point $(t,0)$ which is $\mathbb{F}_q$-rational then $\mathcal{X}_5$ is isomorphic over $\mathbb{F}_q$ to a curve of type $\mathcal{C}_a$ in \eqref{family}.
\endproof

\section{Other examples of maximal curves}\label{Section:HigherGenus}
In this section we investigate other examples of maximal curves.


\begin{definition}\label{CurveGenere10}
Let $p$ be an odd prime and $b \in \mathbb{F}_p$. We denote by $U_{p,b}$ the curve of the affine equation 
\begin{equation}\label{CurveGenus10}
 x^6 + y^6 + 1 + b x^2y^2 = 0.
 \end{equation}
\end{definition}

\begin{proposition} Let $p\geq5$, then $U_{p,b}$ has genus $10$ and $Jac(U_{p,b})\simeq_{\mathbb{F}_{p^2}}  E_1^3\times E_2\times E_3^3\times E_4^3$, where

\begin{eqnarray*}
E_1& :& y^2 = x^3 + 3(b + 6)x^2 + 3(b + 3)(b + 12)x + 27(b + 3)^2,\\
E_2 &:& y^2 = x^3 - 3^3b^2x^2 + 2^33^3b(b^3 + 27)x - 3^32^4(b^3 + 27)^2,\\
E_3 &: & y^2 = x^3 + 2bx^2 + b^2x - 2^2,\\
E_4 &: &y^2 = x^3-48b^2x^2+768b(b^3 + 27)x - 4096(b^3 + 27)^2.\\
\end{eqnarray*}
\end{proposition}

\proof  
By \cite[Proposition 10]{Kawakita}, $Jac(U_{p,b})\simeq_{\mathbb{F}_{p^2}}  H_1^3\times H_2\times Jac(H_3)^3$, where
\begin{eqnarray*}
H_1&:& y^2 = ((-b-3)x + 3)(1 -3x(1 - x)),\\
H_2 &:& x^3 + y^3 + 1 + bxy = 0,\\
H_3 &:& y^2 = -(x^3 + bx^2 + 4)(x^3 + 1).\\
\end{eqnarray*}

By \cite[Page 27]{Djuka}, for $G_a : y^2 = (x^3 + ax^2 -4)(-16x^3 + 16)$, $a\in \mathbb{F}_p$,  we have that $Jac(G_a)\simeq_{\mathbb{F}_{p^2}} E^{\prime}_3\times E^{\prime}_4$, where
\begin{eqnarray*}
E^{\prime}_3&:& y^2 =x^3-a/2+a^2/16-1/16,\\
E^{\prime}_4 &:& y^2=16(a^3-27)x^3-48a^2x^2+48ax-16.\\
\end{eqnarray*}
This shows that $Jac(H_3)\simeq_{\mathbb{F}_{p^2}} E_3\times E_4$.
Finally, it turns out that $H_1$ is birational to $E_1$, $H_2$ is birational to $E_2$. Then the claim follows.
\endproof

\begin{proposition}\label{Prop:Genus10}
The curve $U_{p,b}$ in \eqref{CurveGenus10} is $\mathbb{F}_{p^2}$-maximal  when 
$$(p, b) \in\{(89, 58), (101, 96), (131, 100), (191, 116),(227, 69), (239, 94), (251, 3)\}.$$
\end{proposition}
\proof
Direct computations done using the package MAGMA \cite{MAGMA} for $ p \le 251$.
\endproof

\section*{Acknowledgments}

The research of D. Bartoli and  M. Giulietti was partially supported by the Italian National Group for Algebraic and Geometric Structures and their Applications (GNSAGA - INdAM). D. Bartoli was supported by {\em Universit\`a degli Studi di Perugia - Fondo ricerca di base Esercizio 2015 - Project: Codici Correttori di Errori}. M. Giulietti was supported by {\em Universit\`a degli Studi di Perugia - Fondo ricerca di base  Esercizio 2015 - Project: Geometrie di Galois, Curve Algebriche su campi finiti e loro Applicazioni}. M. Kawakita was partially supported by JSPS Grant-in-Aid for Scientific Research (C) 17K05344.

\bibliographystyle{abbrv}

\end{document}